\documentclass[12pt, twoside]{article}
\usepackage{times}
\usepackage{enumerate}
\def\titlerunning#1{\gdef\titrun{#1}}
\makeatletter
\def\author#1{\gdef\autrun{\def\and{\unskip, }#1}\gdef\@author{#1}}
\def\address#1{{\def\and{\\\hspace*{18pt}}\renewcommand{\thefootnote}{}%
\footnote {#1}}%
\markboth{\autrun}{\titrun}}
\makeatother
\def\email#1{e-mail: #1}
\def\subjclass#1{{\renewcommand{\thefootnote}{}%
\footnote{\emph{Mathematics Subject Classification (2010):} #1}}}
\def\keywords#1{\par\medskip
\noindent\textbf{Keywords.} #1}

\usepackage{amsmath}
\usepackage{amsfonts}
\usepackage{amssymb}
\usepackage[english]{babel}
\usepackage{graphicx}
\usepackage{amsthm}
\usepackage{xcolor}

\newtheorem{thm}{Theorem}[section]
\newtheorem{cor}[thm]{Corollary}
\newtheorem{lem}[thm]{Lemma}
\newtheorem{prop}[thm]{Proposition}
\newtheorem{ex}[thm]{Example}

\newtheorem{prob}[thm]{Problem}
\usepackage{wasysym}

\usepackage{bbm}

\theoremstyle{definition}
\newtheorem{defn}[thm]{Definition}

\theoremstyle{remark}
\newtheorem{rem}[thm]{Remark}

\numberwithin{equation}{section}

\newcommand{\D}{\mathbb{D}}
\newcommand{\F}{\mathbb{F}}
\newcommand{\U}{\mathbb{U}}
\newcommand{\Z}{\mathbb{Z}}

\newcommand{\FF}{\mathcal{F}}

\begin{document}
\baselineskip=17pt
\titlerunning{Highly symmetric KTSs whose orders fill a congruence class}
\title{The first families of highly symmetric Kirkman Triple Systems whose orders fill a congruence class}

\author{Simona Bonvicini 
\and 
Marco Buratti
\and
Martino Garonzi
\and
Gloria Rinaldi
\and
Tommaso Traetta
}

\date{}

\maketitle

\address{
S. Bonvicini: Dipartimento di Scienze Fisiche Informatiche Matematiche, Universit\`a di Modena e Reggio Emilia, via Campi 213/B I-41100 Modena, Italy; \email{simona.bonvicini@unimore.it}
\and
M. Buratti: Dipartimento di Matematica e Informatica, Universit\`a di Perugia, via Vanvitelli I-06123 Perugia,  Italy; \email{buratti@dmi.unipg.it}
\and
M. Garonzi: Departamento de Matem\'atica, Universidade de Bras\'ilia, Campus Universit\'ario Darcy Ribeiro, Bras\'ilia (DF - Brazil); \email{mgaronzi@gmail.com}
\and
G. Rinaldi: Dipartimento di Scienze e Metodi dell'Ingegneria,
Universit\`a di Modena e Reggio Emilia, viale Amendola 2, I-42122 Reggio Emilia, Italy;\email{gloria.rinaldi@unimore.it}
\and
Tommaso Traetta: DICATAM-Sezione Matematica, Universit\`a degli Studi di Brescia, via Branze 38, I-25123 Brescia Italy; \email{tommaso.traetta@unibs.it}
}

\subjclass{Primary 	05B07; Secondary 20B25}

\begin{abstract}
Kirkman triple systems (KTSs) are among the most popular combinatorial designs and their existence has been settled a long time ago. Yet,
in comparison with Steiner triple systems, little is known about their
automorphism groups.  In particular, there is no known congruence class representing the orders of
a KTS with a number of automorphisms at least close to the number of points. 
We fill this gap by proving that whenever $v \equiv 39$ (mod 72), or $v \equiv 4^e48 + 3$ (mod $4^e96$) and $e \geq 0$,
there exists a KTS on $v$ points having at least $v-3$ automorphisms.

This is only one of the consequences
of a careful investigation on the KTSs  with
an automorphism group $G$ acting sharply transitively on all but three points. 
Our methods are all constructive and yield KTSs which in many cases inherit 
some of the automorphisms of $G$, thus increasing the total number of symmetries.

To obtain these results it was necessary to introduce new types of difference families
(the doubly disjoint ones) and difference matrices (the splittable ones)
which we believe are interesting by themselves.

\keywords{Steiner triple system; Kirkman triple system; group action; difference family; difference matrix.}
\end{abstract}

\section{Introduction}
Steiner and Kirkman triple systems are undoubtedly amongst the most popular discrete structures.
A {\it Steiner triple system} of order $v$, briefly STS$(v)$, is a pair $(V,{\cal B})$ where $V$ is a set of $v$ points and
$\cal B$ is a set of 3-subsets ({\it blocks}) of $V$ with the property that any two distinct points are contained in
exactly one block. A {\it Kirkman triple system} of order $v$, briefly KTS$(v)$, is an STS$(v)$ together with a {\it resolution} $\cal R$ of its block-set $\cal B$,
that is, a partition of ${\cal B}$ into classes ({\it parallel classes}) each of which is, in its turn, a partition of the point-set $V$.
It has been known since the mid-nineteenth century that
a STS$(v)$ exists if and only if $v\equiv 1$ or 3 (mod 6) \cite{Kirkman47}. 
The analogous result for KTSs has been instead obtained more than a century later 
\cite{RCW}: a KTS$(v)$ exists if and only if $v\equiv 3$ (mod 6). 

An automorphism of a STS is a permutation of its points leaving the block-set invariant.
Analogously, an automorphism of a KTS is a permutation of its points leaving the resolution invariant.
Thus an automorphism of a KTS is automatically an automorphism of the underlying STS though the converse is not true in general.

For general background on these topics we refer to \cite{CR}.

\medskip
In the next section we will give a brief survey of the automorphism groups of STSs and KTSs.
Looking at the results it will be evident how little we know about KTSs in comparison with STSs. For instance,
it is has been known for close to a century that there is an STS$(v)$ with an automorphism group of order $v$ for any admissible $v$.
Yet, we are still quite far from a similar result for KTSs. At the moment the existence of a KTS$(v)$ with an automorphism group of order
$v$ or $v-1$ is known only when $v$ has a special prime factorization. In particular, there is no known 
congruence $v\equiv k$ (mod $n$) which guarantees the existence of a KTS$(v)$ with an automorphism group
of order close to $v$. 

Adopting the combinatorial-analog of the famous {\it Erlangen program} by Felix Klein \cite{Hawkins}, 
we believe that the interest of a discrete structure is proportional to the number of its automorphisms. 
Motivated by this and by the shortage of results mentioned above, in this paper we deeply investigate Kirkman triple systems which are {\it $3$-pyramidal}, i.e., admitting an automorphism group acting sharply 
transitively on all but three points. 

Now we make a short digression to explain why adopting this ``philosophy" also brings practical benefits.
Given the definition of a discrete structure, the first natural target, which could be very difficult, is to determine under which constraints it exists. The second 
is to give an explicit construction for this structure; in some cases this could be even more difficult. 
For instance, thanks to the recent seminal work of P. Keevash \cite{K}, it is known that a {\it Steiner $t$-design} exists provided that the trivial necessary conditions 
are satisfied and that its order is sufficiently large (the case $t=2$, due to R.M. Wilson \cite{W1,W2,W3}, dates back to the 70's). This is an outstanding achievement which 
seemed completely out of reach only a decade ago; yet, the probabilistic methods used by Keevash are non-constructive,
and do not provide an explicit lower bound on the order of a $t$-design that guarantees its existence.
On the contrary, the request to have many symmetries, besides being in compliance with the Erlangen program, allows to develop constructive algebraic methods. 
This paper gives the complete recipe to construct, explicitly, a KTS$(v)$ for any $v$ as in (i), (ii), (iii) of our main result below.

\begin{thm}\label{main}
A necessary condition for the existence of a $3$-pyramidal KTS$(v)$ is that $v=24n+9$ or $v=24n+15$ or
$v=48n+3$ for some $n$ which, in the last case, must be of the form $4^em$ with $m$ odd. 
This condition is also sufficient in each of the following cases:
\begin{itemize}
\item[\rm(i)] $v=24n+9$ and $4n+1$ is a sum of two squares;
\item[\rm(ii)] $v=24n+15$ and either $2n+1\equiv0$ $($mod $3)$ or the square-free part of $2n+1$ does not have any prime $p\equiv11$ $($mod $12)$;
\item[\rm(iii)] $v=48n+3$.
\end{itemize}
\end{thm}
In particular, (ii) and (iii) allow us to reach our main target of getting some families of highly symmetric KTSs 
whose orders fill a congruence class.
\begin{cor}
There exists a $3$-pyramidal KTS$(v)$ for all $v\equiv 39$ $($mod $72)$,  and 
for all $v\equiv4^e48+3$ $($mod $4^e96)$ whatever the non-negative integer $e$ is.
\end{cor}

We point out that in many cases the constructed 3-pyramidal KTSs inherit
some  automorphisms of the group $G$ acting sharply transitively on all but
three points. This permits to increase their number of symmetries
 considerably (see Remarks \ref{rem(i)}, \ref{rem(ii)} and \ref{rem(iii)}).

After the brief survey of Section 2 concerning some results on the automorphism groups of Steiner and Kirkman triple systems, 
the article will be structured as follows.
In Section 3 we provide the difference methods to construct a 3-pyramidal KTS and
we prove that each group having a 3-pyramidal action on a KTS fixes one parallel class
and acts transitively on the remaining ones. We also prove that such a group must have  
exactly three involutions, and these involutions are pairwise conjugate. Groups
with this property will be called {\it pertinent}. 
Although in literature there is no lack of articles 
on groups with three involutions (see, for example, \cite{Janko, Konvisser}), none of them
allows us to determine the set of ``relevant'' orders.
We prove (Theorem \ref{PertinentOrders}) that such orders
are precisely those of the form $12n + 6$ or $4^\alpha (24n + 12)$,
and from this we partially derive the necessary condition in Theorem  \ref{main}. 
The proof of the
``only if part'' of Theorem \ref{PertinentOrders} is purely group theoretical and for this reason the whole Section 4 is dedicated to it. 
However,
its reading is  not necessary for understanding
the rest of the article, whose nature is purely combinatorial.

The constructive part of Theorem \ref{main} will be proven in Section 11 and it is the result
of numerous direct constructions (Sections 6, 7 and the appendix) and recursive constructions (Sections 8, 9, 10).  These results are preceded
by a brief section (Section 5) useful for understanding the notation and terminology 
used throughout the rest of the paper.
The recursive constructions required the introduction of new concepts such as
{\it doubly disjoint difference family} and {\it splittable difference matrix}
which we believe may be important by themselves.

The article concludes (Section 12) with a short list of open problems.

\section{A brief survey of the automorphism groups of Steiner and Kirkman triple systems}
\noindent

The literature on Steiner and Kirkman triple systems having an automorphism with a prescribed property or an automorphism group
with a prescribed action is quite extensive. 

For instance the set of values of $v$ for which there exists an STS$(v)$ with an involutory automorphism fixing
exactly one point ({\it reversed} STS) has been established in \cite{Doyen,Rosa,T}: it exists if and only if
$v\equiv 1, 3, 9$ or 19 (mod 24).
Results concerning the {\it full} automorphism group of a STS have been obtained by Mendelsohn \cite{Mend}
and Lovegrove \cite{L}.

Here, we just provide a brief survey of what is known on the existence of systems 
whose automorphisms are at least close to the number of points.

Adopting a terminology coined by E. Mendelsohn and A. Rosa \cite{MendRosa}, we say that a 
combinatorial design is {\it $f$-pyramidal}  if it admits an automorphism group $G$ fixing
$f$ points and acting sharply transitively on the others. If $f=0$ one usually speaks of a
{\it regular} design and, more specifically, of a {\it cyclic design} if the group $G$ is cyclic.
If $f=1$ one usually speaks of a {\it $1$-rotational design}.

It was proved a long time ago \cite{P} that there exists a cyclic STS$(v)$ for all admissible values
of $v$ except $v=9$. On the other hand the unique STS(9), that is the point-line design
associated with the affine plane of order 3, is clearly regular under the action of $\Z_3^2$.
Thus there exists a regular STS$(v)$ for all admissible values of $v$.

The analogous problem of determining the set of values of $v$ for which there exists a regular KTS$(v)$
is almost completely open and it appears to be very difficult. The few known results on this problem are the following. 
The parallel classes of the point-line design associated with the $n$-dimensional 
affine space over the field of order 3 clearly give a KTS$(3^n)$ that is regular under the action of $\Z_3^n$. A necessary condition
given in \cite{MR} for the existence of a cyclic KTS$(6n+3)$ is that $2n+1$
is not a prime power congruent to 5 (mod 6). This condition is also sufficient up to $n=32$ \cite{MR} 
and when all prime factors of $n$ are congruent to 1 (mod 6) \cite{GMJ}.

The existence of a $1$-rotational STS$(v)$ has been thoroughly investigated in
\cite{BBRT, B, M, PR} leaving the problem open only for the orders $v$
satisfying, simultaneously, the following conditions: $v=(p^3-p)n+1\equiv1$ (mod 96) with $p$ a prime;
$n\not\equiv0$ (mod 4); the odd part of $v-1$ is square-free and without prime factors $\equiv1$ (mod 6).

The $1$-rotational KTSs have a very nice structure. Indeed a group with a 1-rotational action on them
(necessarily {\it binary}, i.e., admitting exactly one involution) is transitive on the parallel classes.
On the other hand, as in the regular case, very little is known about their existence which has 
been proved only for orders $v$ of the the following types: $v$ is a power of $3$ (the already mentioned
regular KTS$(3^n)$ is also 1-rotational); all prime factors of ${v-1\over2}$ are congruent to $1$ (mod $12$) \cite{BZ};
$v=8n+1$ with all the prime factors of $n$ congruent to $1$ (mod $6$) \cite{BZ01}.

The existence problem for 3-pyramidal STSs was completely settled in \cite{BRT}: there exists
a $3$-pyramidal STS($v$) if and only if $v\equiv 7,9,15$ (mod $24$) or $v\equiv 3,19$ (mod $48$).

As an obvious consequence of the above result, a 3-pyramidal KTS$(v)$ may exist only when $v\equiv9$ (mod 24) or $v\equiv15$ (mod 24) 
or $v\equiv3$ (mod 48). 
The main result of this paper (Theorem \ref{main}) provides, in particular, a complete answer in the last case $v\equiv3$ (mod 48). 

Finally, a STS$(v)$ is called  {\it $1$-transrotational} if it has an
automorphism group $G$ that fixes exactly one point, switches two points,
and acts sharply transitively on the remaining $v-3$. This terminology was
first used in \cite{Gardner} under the assumption that $G$ is cyclic,
though $G$ just need to be binary. One cannot fail to notice a certain kinship
between 3-pyramidal and 1-transrotational STSs, but apart from the fact
that their groups are deeply different, the sets of orders for which they
exist do not coincide. Indeed it
was proved in \cite{Gardner} that a 1-transrotational STS$(v)$ under the
cyclic group exists if and only if $v\equiv 1, 7, 9$ or $15$ (mod 24). It
is easy to check that the same holds if we remove the assumption that the
group be cyclic.
As far as we are aware, nobody studied 1-transrotational KTS$(v)$.
Considering the above, they might exist only for $v\equiv 9$ or 15 (mod 24)
but it is not difficult to exclude the case $v\equiv15$ (mod 24). This will
be shown in a paper in preparation \cite{BN} where we will deal with the
case $v\equiv 9$ (mod 24).

\section{Difference families and $3$-pyramidal Kirkman\break triple systems}\label{DFs}
In this section we show that the existence of a $3$-pyramidal KTS over a group $G$
is equivalent to constructing a suitable {\it difference family} (DF) in $G$ relative to a {\it partial spread}, a concept introduced by the second author in \cite{B}.
We point out that throughout the paper, except for Section \ref{pertinent}, every group will be denoted additively.

A partial spread of a group $G$ is a set $\Sigma$ of subgroups of $G$ whose mutual intersections are trivial.
If $\tau=\{d_1^{e_1}, \dots, d_n^{e_n}\}$ is the multiset
(written in ``exponential" notation) of the orders of all subgroups belonging to $\Sigma$,
we say that $\Sigma$ is of {\it type} $\tau$ or a $\tau$-partial spread.
A {\it spread} (or {\it partition}) of $G$ is a partial spread whose members between them cover the whole group $G$.

The {\it list of differences} of a triple $B=\{x,y,z\}$ of elements of $G$ is the multiset $\Delta B$ of size 6 defined by 
$$\Delta B=\{\pm(x-y),\pm(x-z),\pm(y-z)\}.$$
The list of differences of a family $\FF$ of triples of $G$, denoted by $\Delta\FF$, is the multiset union of the lists
of differences of all its triples.
Also, the {\it flatten} of ${\cal F}$, denoted by $\Phi(\FF)$, is the multiset union of all the triples of $\FF$.
$$\Delta\FF=\bigcup_{B\in\FF}\Delta B;\quad\quad \Phi(\FF)=\bigcup_{B\in\FF} B.$$

A $(G, \Sigma, 3, 1)$ difference family (DF) is a family $\FF$ of triples of $G$
({\it base blocks}) whose list of differences is the set of all elements of $G$ not belonging to any member of the partial spread $\Sigma$.
If $\Sigma=\{H\}$ we write $(G,H,3,1)$-DF or simply $(G,3,1)$-DF when $H=\{0\}$. 
If $\Sigma$ is a partial spread of type $\tau$, we also use the notation $(G, \tau, 3, 1)$-DF.

If ${\cal F}$ is a $(G,H,3,1)$-DF, then its size is clearly equal to ${|G\setminus H|\over 6}$ and then its flatten $\Phi({\cal F})$ has size 
${|G\setminus H|\over 2}$. Thus, if $J$ is a subgroup of $H$ of order 2, one can ask whether $\Phi({\cal F})$ is a complete system of representatives
for the left cosets of $J$ that are not contained in $H$. In the affirmative case we say that ${\cal F}$ is $J$-resolvable.

\begin{defn}\label{rdf}
Let $\cal F$ be a $(G,H,3,1)$-DF and let $J$ be a subgroup of $H$ of order 2. We say that $\cal F$ is {\it $J$-resolvable} if
its flatten is a complete system of representatives for the left cosets of $J$ in $G$ that are not contained in $H$. So, equivalently, if
we have $\Phi(\FF)+J=G\setminus H$.
\end{defn}

We note that the development of a $J$-resolvable $(G,H,3,1)$-DF is a {\it Kirkman frame} \cite{S} admitting $G$ as a sharply point transitive 
automorphism group.

A {\it multiplier} of a $J$-resolvable $(G,H,3,1)$-DF, say $\cal F$, is an automorphism 
$\mu$ of $G$ leaving $\cal F$ invariant. We say that $\mu$ is a {\it strong} multiplier if it fixes
$H$ element-wise.

The following fact is straightforward. 

\begin{prop}\label{matrioska}
Let $G=H_n\geq \dots \geq H_1 \geq J$ be a chain of subgroups of $G$ with $J$ of order $2$.
If there exists a $J$-resolvable $(H_{i+1},H_i,3,1)$-DF, say ${\cal F}_i$, 
for $1\leq i\leq n-1$, 
then ${\cal F} = \displaystyle\bigcup_{i=1}^{n-1} {\cal F}_i$
is a $J$-resolvable $(G,H_1,3,1)$-DF.

Furthermore, 
${\cal F}$ inherits the strong multipliers of ${\cal F}_{n-1}$.
\end{prop}

Difference families are a crucial topic in Design Theory \cite{BJL,CD}. In particular, as a special case of 
Theorem 2.1 in  \cite{BRT}, it is possible to characterize the $3$-pyramidal STS($6n+3$) in
terms of difference families as follows.

\begin{thm}
There exists a $3$-pyramidal STS$(6n+3)$ if and only if there exists a $(G,$ $\{2^3, 3^e\}, 3, 1)$-DF for a suitable group 
$G$ of order $6n$ with exactly three involutions, and a suitable integer $e$.
\end{thm}

In the following lemma we recall what the 3-pyramidal STS$(6n+3)$ generated by a $(G,\{2^3,3^e\},3,1)$-DF looks like.
\begin{lem}\label{description}
Up to isomorphism, $(V,{\cal B})$ is a STS$(6n+3)$ that is $3$-pyramidal under $G$ if and only if
the following facts hold: 
\begin{itemize}
\item[$(i)$] $V=G \ \cup \ \{\infty_1,\infty_2,\infty_3\}$;
\item[$(ii)$] the action of $G$ on $V$ is the addition on the right with the rule that $\infty_i+g=\infty_i$ for each $g\in G$; 
\item[$(iii)$] $G$ has exactly three involutions, say $j_1$, $j_2$, $j_3$;
\item[$(iv)$] a system of representatives for the $G$-orbits on $\cal B$ is of the form $$\{B_\infty,B_1,B_2,B_3\} \ \cup \ {\cal H} \ \cup \ {\cal F}$$ where:
\begin{description}
\item $B_{\infty}=\{\infty_1,\infty_2,\infty_3\}$; \item $B_i=\{\infty_i,0,j_i\}$ for $i=1,2,3$; \item $\cal H$ is a set of $e$ subgroups of $G$ of order $3$;
\item $\cal F$ is a $(G,\Sigma,3,1)$-DF with $\Sigma=\{\{0,j_1\},\{0,j_2\},\{0,j_3\}\} \cup {\cal H}$.
\end{description}
\end{itemize}
\end{lem}

\begin{defn} 
Throughout this paper a finite group $G$ will be called ``pertinent'' if it has precisely three involutions, and the three of them are pairwise conjugate in $G$. 
\end{defn}

Obviously, a pertinent group is necessarily non-abelian. Up to isomorphism, the smallest pertinent group is $\D_6$, i.e., the dihedral group of order 6. 
The next one is $\mathbb{A}_4$, the alternating group of degree 4.

As already said, we prefer to write every group in additive notation. So, differently from the mostly used representation, we prefer to see $\D_6$
as the additive group $D$ with underlying set $\Z_2\times\Z_3$ and operation law $ \hat{+}$ defined by
$(a,b) \  \hat{+} \ (c,d)=(a+c,(-1)^{c}b+d)$. Adopting this representation, it is easy to see that the difference $ \hat{-}$ in $D$ works 
as follows: $$(a,b) \  \hat{-} \ (c,d)=(a-c,(-1)^{c}(b-d)).$$ 
The three involutions of $D$ are $(1,0)$, $(1,1)$ and $(1,2)$. 
The fact that $$(1,1) \ \hat{+} \ (1,0) \ \hat{-} \ (1,1)=(1,2)\quad\quad{\rm and}\quad\quad (1,2) \ \hat{+} \ (1,0) \ \hat{-} \ (1,2)=(1,1)$$
confirms that $D$ is pertinent.

The alternating group $\mathbb{A}_4$ will be also represented additively as the first term of an infinite series of
pertinent additive groups that we will construct in the proof of the ``if part" of Theorem \ref{PertinentOrders}.

The crucial ingredient to characterize and construct the $3$-pyramidal KTSs are
some special $(G,\{2^3,3\},3,1)$-DFs with $G$ pertinent defined as follows.

\begin{defn}\label{JRDF}
Let $\cal F$ be a $(G,\{2^3,3\},3,1)$-DF with $G$ pertinent and let $J$ be a subgroup of $G$ of order $2$. 
We say that $\cal F$ is {\it $J$-resolvable} if there exists $a, b\in G$ such that
\begin{itemize}
\item $J$, $a+J-a$ and $b+J-b$ are the three subgroups of order 2 of $G$;
\item $\Phi(\FF)\ \cup \ \{0, a, b\}$ is a complete system of
representatives for the left cosets of $J$ in $G$.
\end{itemize}
\end{defn}

The following fact is straightforward. 

\begin{prop}\label{matrioska2}
Let $J$ be a subgroup of order $2$ of a pertinent group $G$ and let $H$ be a pertintent subgroup of $G$ containing $J$.
If $\FF_0$ is a $J$-resolvable $(H,\{2^3,3\},3,1)$-DF and $\cal F$ is a $J$-resolvable $(G,H,3,1)$-DF,
then ${\cal F} \ \cup \ {\cal F}_0$ is a $J$-resolvable $(G,\{2^3,3\},3,1)$-DF.
\end{prop}


Speaking of a $(G,\{2^3,3\},3,1)$-RDF with $G$ pertinent, we will mean a $J$-resolvable  $(G,\{2^3,3\},3,1)$-DF
with $J$ one of the three subgroups of $G$ of order 2. 

The following result gives a characterization of the $3$-pyramidal KTSs and, more importantly, a way to construct them.

\begin{thm}\label{DF}
There exists a $3$-pyramidal KTS$(6n+3)$ with $n > 0$, if and only if there exists a
$(G,\{2^3,3\},3,1)$-RDF for a suitable pertinent group $G$
of order $6n$.
\end{thm}
\begin{proof} 
($\Longrightarrow$). \quad
Let $(V,{\cal B},{\cal R})$ be a KTS$(6n+3)$ that is $3$-pyramidal under $G$.
Thus $(V,{\cal B})$ is a STS$(6n+3)$ that is $3$-pyramidal under $G$, hence
we can assume that $V$ and $\cal B$ satisfy the conditions listed in Lemma \ref{description}.

Let $\cal P$ be the parallel class containing the block $B_{\infty}$. 
Obviously, we have ${\cal P}+g={\cal P}$ for each $g\in G$. 
It follows that the distinct right translates of the block $H$ 
of $\cal P$ through 0 form a partition of $G$. This 
clearly implies that $H$ is a subgroup of $G$ and
that the blocks of $\cal P$ are $B_\infty$ and all
the right cosets of $H$ in $G$.

Now let $\cal Q$ be the parallel class of $\cal R$ containing the block $B_1=\{\infty_1,0,j_1\}$.
It is not difficult to see that the $G$-stabilizer of ${\cal Q}$ coincides with the $G$-stabilizer of $B_1$ that is $J:=\{0,j_1\}$.
Thus, for $i=2,3$, the block of ${\cal Q}$ through $\infty_i$ is of the form $\{\infty_i,g_i,g_i+j_1\}$ 
for a suitable group element $g_i$. This block necessarily belongs to the orbit of $B_i$,
hence we have $\{\infty_i,0,j_i\}+t_i=\{\infty_i,g_i,g_i+j_1\}$ for a suitable $t_i$.
This equality easily implies that $j_i=g_i+j_1-g_i$, hence the three involutions $j_1$, $j_2$, $j_3$ 
are pairwise conjugate, i.e., $G$ is pertinent.

The fact that the $G$-stabilizer of $\cal Q$ is $J$ also implies that the $2n-2$ triples of ${\cal Q}$ not containing the ``points at infinity" 
can be grouped into pairs $\{A_i,A_i+j_1\}$, $1\leq i\leq n-1$, and that the $G$-orbit of $\cal Q$ has length ${|G|\over2}=3n$.
Then, given that the resolution of a KTS$(6n+3)$ has size $3n+1$, we deduce that ${\cal R}=\{{\cal P}\} \ \cup \ Orb({\cal Q})$.
Also, if we set ${\cal F}=\{A_i \ | \ 1\leq i\leq n-1\}$, we can claim that
a set of {\it base blocks} for ${\cal B}$ is given by 
$$\{B_\infty, \ B_1, \ B_2, \ B_3, \ H\} \ \cup \ {\cal F}.$$
It follows, by condition $(iv)$ in Lemma \ref{description}, that $\cal F$ is a $(G,\{2^3,3\},3,1)$-DF.

Given that the blocks of $\cal Q$ partition $V$, we have $(\Phi({\cal F}) \ \cup \ \{0,g_2,g_3\})+J=G$. 
This means that $\Phi({\cal F}) \ \cup \ \{0,g_2,g_3\}$ is a complete system
of representatives for the left cosets of $J$ in $G$, i.e., $\cal F$ is $J$-resolvable.

($\Longleftarrow$).\quad Let $G$ be a pertinent group of order $6n$ whose involutions are $j_1$, $j_2$, $j_3$,
and assume that $\cal F$ is a $\{0,j_1\}$-resolvable $(G,\{2^3,3\},3,1)$-DF. 
For $i=2,3$, there are suitable group elements $g_2$, $g_3$ such that $j_i=g_i+j_1-g_i$  and $(\Phi({\cal F}) \ \cup \ \{0,g_2,g_3\})+\{0,j_1\}=G$. 
Let $H$ be the subgroup of $G$ of order 3 belonging to the partial spread associated with $\cal F$. 
Two parallel classes of the STS$(6n+3)$ generated by $\cal F$ are clearly the following:
$${\cal P}=\{B_\infty\} \ \cup \ \{\mbox{rights cosets of $H$ in }G\};$$
$${\cal Q}=\bigl{\{}\{\infty_i,g_i,g_i+j_1\} \ | \ i=1,2,3\bigl{\}} \ \cup \ \bigl{\{}A, A+j_1 \ | \ A\in{\cal F}\bigl{\}}.$$
Their $G$-stabilizers are, respectively, $G$ and $\{0,j_1\}$ so that their $G$-orbits have size 1 and $3n$.
It easily follows that $\{{\cal P}\} \ \cup \ Orb({\cal Q})$ is a $G$-invariant resolution of the STS$(6n+3)$
generated by $\cal F$, namely a 3-pyramidal KTS$(6n+3)$.
\end{proof}

In view of the above theorem, it is important to determine the set of {\it pertinent numbers}, i.e., the set 
of orders of the pertinent groups. 

\begin{thm}\label{PertinentOrders}
There exists a pertinent group of order $n$ if and only if $n\equiv6$ $($mod $12)$ or $n=4^\alpha m$ with
$\alpha>0$ and $m\equiv3$ $($mod $6)$.
\end{thm}
\begin{proof}
The proof of the ``only if part" is purely group theoretical and for convenience it is postponed to Section \ref{pertinent}.
Here we prove the ``if part". 

If $n\equiv6$ (mod 12), we have $n=6m$ for a suitable odd integer $m$. Then, recalling that $D$ is pertinent,
it is clear that $D\times H$ is a pertinent group of order $6m$ for every group $H$ of order $m$.

\smallskip
Now let $n=4^\alpha m$ with $\alpha>0$ and $m\equiv3$ $($mod $6)$. 
Consider the matrix  $\Theta=\begin{pmatrix}1& \ \ 0& \ \ 0\cr0& \ \ 0& \ \ 1\cr0&-1&-1\end{pmatrix}$ 
and let $G_\alpha$ be the group with underlying set $\Z_3\times\Z_{2^\alpha}\times\Z_{2^\alpha}$ and operation $\hat{+}$
defined by the rule $$(a,b,c) \ \hat{+} \ (d,e,f)=(a,b,c)\cdot \Theta^d+(d,e,f).$$

This is, up to isomorphism, the outer semidirect product of $\Z_{2^\alpha}^2$ and $\Z_3$ with respect to the group omomorphism
$\theta: \Z_3 \longrightarrow Aut(\Z_{2^\alpha}^2)$ defined by the rule
$$\theta(1)(x,y)=(-y,x-y)\quad\forall (x,y)\in \Z_{2^\alpha}^2.$$ 
The difference of two triples $(a,b,c)$ and $(d,e,f)$ of the group has 
a convenient form; it is the usual difference in the
abelian group $\Z_3\times\Z_{2^\alpha}^2$ multiplied by the inverse of $\Theta^d$:
\begin{equation}\label{ominus}
(a,b,c)\ \hat{-}  \ (d,e,f)=(a-d,b-e,c-f)\cdot \Theta^{-d}.
\end{equation}
More explicitly, we have 
$$(a,b,c) \ \hat{-} \ (d,e,f)=\begin{cases}(a-d, \ b-e, \ c-f) & {\rm if} \ d=0 \cr (a-d, \ e-b+c-f, \ e-b) & {\rm if} \ d=1\cr (a-d, \ f-c, \ b-e+f-c) & {\rm if} \ d=2\end{cases}$$
It is a simple exercise to check that $G_\alpha$ has exactly three involutions that are 
$$(0,2^{\alpha-1},0),\quad\quad (0,0,2^{\alpha-1}), \quad\quad (0,2^{\alpha-1},2^{\alpha-1})$$
and that they are pairwise conjugate. Indeed we have:
$$(1,0,0) \ \hat{+} \ (0,2^{\alpha-1},0) \ \hat{-} \ (1,0,0) \ = \ (0,2^{\alpha-1},2^{\alpha-1});$$
$$(2,0,0) \ \hat{+} \ (0,0,2^{\alpha-1}) \ \hat{-} \ (2,0,0) \ = \ (0,2^{\alpha-1},2^{\alpha-1})$$

Thus $G_\alpha$ is a pertinent group of order $4^\alpha3$. Then, if $H$ is any group of odd order $m$, it is clear
that the direct product $G_\alpha\times H$ is a pertinent group of order $3\cdot4^\alpha m$ whose involutions are $(0,2^{\alpha-1},0,0)$, $(0,0,2^{\alpha-1},0)$ 
and $(0,2^{\alpha-1},2^{\alpha-1},0)$. 
\end{proof}

The alternating group $\mathbb{A}_4$ can be seen, up to isomorphism, as the group $G_1$.

 \section{Pertinent groups}\label{pertinent}

Here we prove the ``only if" part of Theorem \ref{PertinentOrders}.
Considering that the arguments used are purely group theoretical,
the reading of this section can be postponed to a later time without 
compromising the understanding of the rest of the article.
Also, we point out that in this section, unlike the rest of the paper, 
we prefer to denote groups in multiplicative notation.

In the following, let $G$ be a pertinent group and denote by $K$ the subgroup of $G$ generated by the three involutions $i,j,k$. Let $C_G(K)$ be the centralizer of $K$ in $G$, that is
\[
C_G(K)=\{g \in G\ :\ g^{-1}kg=k\ \text{for every $k\in K$}\},
\]
and set $C=C_G(K)$.
Clearly, $K$ is a characteristic subgroup of $G$, so both $K$ and $C$ are normal in $G$. Since $G$ acts on $\{i,j,k\}$ by conjugation with kernel $C$, the quotient $G/C$ is isomorphic to a transitive subgroup of $S_3$, so either $G/C \cong S_3$ or $G/C \cong A_3$ (here $A_3$ denotes the alternating group of degree $3$). Observe in particular that $G$ has an element of order a power of $3$ acting ``cyclically'' on the involutions (meaning that it sends $i$ to $j$, $j$ to $k$ and $k$ to $i$).
Recall that if $H$ is a subgroup of $G$ the ``normalizer'' of $H$ in $G$ is $N_G(H)=\{g \in G\ :\ H^g=H\}$, where $H^g=g^{-1}Hg$ is a ``conjugate'' of $H$ in $G$.

We will make use of the following well-known result in group theory.

\begin{lem}[Frattini's Argument]\label{Frattini}
  If $X$ is a normal subgroup of $G$ and $Q$ is a
  Sylow $p$-subgroup of $X$, then $G=X N_G(Q)$.
\end{lem}

We start providing sufficient conditions for a pertinent group to have subgroups or quotients that are pertinent.

\begin{lem} \label{sub}
Let $H$ be a subgroup of $G$. Then
\begin{enumerate}
\item If $H$ has even order and $HC=G$, then $H$ is pertinent.
\item If $Q$ is a Sylow $p$-subgroup of $C$ and $H=N_G(Q)$, then $H$ is pertinent.
\item If $H$ has even order and contains a Sylow $3$-subgroup of $G$, then $H$ is pertinent.
\item Suppose the involutions of $G$ commute pairwise. If $H$ is normal in $G$ and $|H|$ is odd then $G/H$ is pertinent.
\end{enumerate}
\end{lem}

\begin{proof}
(1) Since $H$ has even order, it contains at least one involution. Considering that $HC=G$, it follows that the action of $H$ on the involutions is transitive. Therefore, $K\leq H$ and  $H$ is pertinent.

(2) By Lemma \ref{Frattini}, we have that $HC=G$. Since $Q\leq C$, it follows that $K$ centralizes $Q$,
hence $K\leq H$, therefore $H$ has even order. By point $(1)$, we obtain that $H$ is pertinent.

(3) Since $|H|$ is even, $H$ contains at least one involution.  Since $|G/C|$ is a multiple of $3$, a Sylow $3$-subgroup $S$ of $G$ is not contained in $C$, hence $S$ acts transitively on $\{i,j,k\}$. Therefore, $H$ contains all three involutions and then it is pertinent.

(4) Suppose $H$ is normal of odd order in $G$ and the involutions commute pairwise. Denote by $i,j,k$ the involutions of $G$. Since $ij=ji$, the element $ij$ is an involution distinct from $i$ and from $j$ so $ij=k$ and the elements $iH,jH,kH \in G/H$ are involutions of $G/H$ and $G/H$ acts transitively by conjugation on them (because $G$ acts transitively by conjugation on $i,j,k$), moreover they are pairwise distinct, for example $iH \neq jH$ because $i^{-1}j = ij = k \not \in H$. We are left to show that $G/H$ has precisely three involutions. If $xH$ is an involution of $G/H$ then $x^2 \in H$ so $x$ has order $2t$ with $t$ odd (being $|H|$ odd), $x^t$ is an involution of $G$, and $xH = (xH)^t = x^t H$. This means that the involutions of $G/H$ are of the form $yH$ with $y$ an involution in $G$, so they are precisely $iH$, $jH$, $kH$ and we deduce that $G/H$ is pertinent.
\end{proof}

For a pertinent group of order $2^n \cdot d$ with $d$ odd, the following two lemmas give us sufficient or necessary conditions for $n$ to be even or odd.

\begin{lem} \label{even}
Suppose $G$ has a normal $2$-subgroup $H$ of order $2^m$. Then $m$ is even.
\end{lem}
\begin{proof}
$G$ has an element $g$ of order a power of $3$ acting cyclically on the three involutions. We claim that $g$ does not fix any non-trivial element of $H$. Indeed if $1 \neq h \in H$ is fixed by $g$ then a suitable power of $h$ is an involution fixed by $g$, but $g$ does not fix any involution. This implies that the $\langle g \rangle$-orbits of $H$ distinct from $\{1\}$ have size divisible by $3$ so $2^m = |H| \equiv 1 \mod 3$ therefore $m$ is even.
\end{proof}

\begin{lem} \label{klein}
Suppose $4$ divides $|G|$. Then $K$ is isomorphic to the Klein group $\Z_2 \times \Z_2$. Moreover writing $|G|=2^n \cdot d$ with $d$ odd, if $n$ is even then $G/C \cong A_3$, and if $n$ is odd then $G/C \cong S_3$.
\end{lem}
\begin{proof}
First, we show that $i,j,k$ commute pairwise, so that $K = \{1,i,j,k\} \cong \Z_2 \times \Z_2$. It is enough to show that two involutions, say $i$ and $j$, commute, because then $ij$ has order $2$ so $ij=k$ (it cannot be $ij=i$ nor $ij=j$) therefore $i$ and $j$ commute with $k$. If this is not the case, then $C$ must have odd order (otherwise an involution in $C$ should commute with the others). But then the order of $G/C$ is a multiple of $2^n\ge 4$ contradicting the fact that $G/C$ is isomorphic to either $A_3$ or $S_3$.


Write $|G|=2^n \cdot d$ with $d$ odd. We know that $G/C$ is isomorphic to either $A_3$ or $S_3$. Letting $Q$ be a Sylow $2$-subgroup of $C$, $N=N_G(Q)$ is pertinent by Lemma \ref{sub}, so writing $|Q|=2^m$, $m$ is even by Lemma \ref{even}. If $G/C \cong A_3$ then $Q$ is a Sylow $2$-subgroup of $G$, so $n=m$ is even. Conversely if $G/C \cong S_3$ then $n=m+1$ is odd.
\end{proof}

To prove the main result of this section, we need the following result.

\begin{lem} \label{tech}
Suppose $|G| = 2^n \cdot d$ where $n \geq 3$ is odd. Then $G$ contains a pertinent subgroup $L$ of order $2^n \cdot 3^s$ for some positive integer $s$.
\end{lem}
\begin{proof}
Let $Q$ be a Sylow $2$-subgroup of $C$ and recall that $|G/C|\in\{3,6\}$. Therefore,
if $|Q|=2^m$, then $m=n$ or $n-1$ according to whether $|G/C|=3$ or $6$.
By Lemma \ref{sub}, we have that $N=N_G(Q)$ is a pertinent group, and since
$Q$ is normal in $N$, by Lemma \ref{even} we have that $m$ is even, hence $m=n-1$ (since $n$ is odd by assumption) and $|G/C|=6$. Considering that $Q\leq C\cap N$, and $G=NC$ (by Lemma \ref{Frattini}), it follows that $|N/Q|$ is a multiple of $|N/(C \cap N)| = |NC/C| = |G/C|=6$, hence $2^n$ divides $|N|$.

Since $|N/Q|\equiv 2 \pmod{4}$, there exists a normal subgroup $D/Q$ of $N/Q$ of index $2$. In particular, all  the Sylow $3$-subgroups of $N/Q$ are contained in $D/Q$, so the number of Sylow $3$-subgroups of $N/Q$ is odd. Let $P/Q$ be a subgroup of $N/Q$ of order $2$. Since $P/Q$ acts by conjugation on the family consisting of the Sylow $3$-subgroups of $N/Q$, there exists one of them, say $H/Q$, normalized by $P/Q$. This implies that $PH/Q = (P/Q)(H/Q) \leq N/Q$ hence $L= PH \leq N$. Moreover $|L| = |Q| \cdot |L/Q| = |Q| \cdot |P/Q| \cdot |H/Q| = 2^n \cdot 3^s$ where $3^s=|H/Q|$. Now $L$ has even order and contains a Sylow $3$-subgroup of $N$, hence $L$ is pertinent by Lemma \ref{sub}.
\end{proof}

We are now ready to prove the ``only if part" of Theorem \ref{PertinentOrders}.

\begin{thm}\label{nodd}
If $G$ is a pertinent group of order $2^nd$ with $d$ odd and $n \geq 2$, then $n$ is even.
\end{thm}
\begin{proof}
We prove the result by contradiction. Let $G$ be a counterexample of minimal order, that is, assume that there exists a pertinent group $G$ such that $|G|= 2^n \cdot d$ is minimal with respect to the property that both $n$ and $d$ are odd, and $n\geq3$.
By Lemma \ref{tech}, we have that $G$ has a pertinent subgroup of order $2^n \cdot 3^s$ for some positive integer $s$, so by the minimality of $|G|$ we must have $d=3^s$.
Recall that $K \cong C_2 \times C_2$ and $G/C \cong S_3$ by Lemma \ref{klein}.

Let $S$ be a Sylow $3$-subgroup of $C$ and note that $|S|=3^{s-1}$ since $|G/C|=6$. Now, by Lemma \ref{sub} we have that $N=N_G(S)$ is pertinent; in particular, $K\leq N$ hence $4$ is a divisor of $|N|$.
Also, since $CN=G$ (by Lemma \ref{Frattini}) and $C\ \cap\ N = C_{N}(K)$, it follows that $N/C_{N}(K) = N/(C\ \cap\ N)\cong NC/C = G/C \cong S_3$. Considering that $S$ is a Sylow $3$-subgroup of $C_{N}(K)$ of order $3^{s-1}$, we have that $|N|= 2^m \cdot 3^s$, for some $2 \leq m \leq n$. Also, by Lemma \ref{klein}, it follows that $m$ is odd. Finally,
since $S$ is a normal subgroup of $N$ of odd order, Lemma \ref{sub} guarantees that
$N/S$ is pertinent. Since $|N/S| = 2^m \cdot 3$ and $m\geq 3$ is odd, by the minimality of $|G|$ we must have $s=1$, hence $|G|= 2^n \cdot 3$.

We are reduced to the case $|G|=2^n \cdot 3$. Let $H:=C_G(i)$ be the centralizer of $i$ in $G$. We have
$|G:H|=3$ because $i$ has $3$ conjugates in $G$, in other words $|H|=2^n$ and $H$ is a Sylow $2$-subgroup of $G$. Let $P$ be a Sylow $3$-subgroup of $G$, then $P=\langle g \rangle$ is a cyclic group of order $3$.
Clearly $g \not \in H$ because $|H|=2^n$, so $P$ acts transitively on the three involutions of $G$.
We prove that $N_G(P)=P$.
Suppose for a contradiction that $N_G(P) \neq P$, then there is an involution, say $i$, normalizing $P$, hence $\langle P, i\rangle$ is a group of order $6$ containing all three involutions, hence
$K\leq  \langle P, i\rangle$, which is a contradiction since $|K|=4$. So $N_G(P)=P$. We now prove that $H$ is normal in $G$. Since $N_G(P)=P$, the subgroup $P$ has precisely $|G:P|=2^n$ conjugates in $G$ therefore $G$ has precisely $(|P|-1) \cdot 2^n = 2 \cdot 2^n$ elements of order $3$. Since $|G|=3 \cdot 2^n$ we deduce that the number of elements of $G$ of order not equal to $3$ is $2^n$ hence there is room in $G$ for only one Sylow $2$-subgroup. We deduce that the Sylow $2$-subgroups are normal hence $H \unlhd G$, so $n$ is even by Lemma \ref{even}. This is a contradiction and the result is proved.
\end{proof}

\section{Notation and terminology}\label{notation}

 A maximal prime power divisor of any integer $n$ will be called a {\it component} of $n$.
As it is standard, given a prime power $q$, we denote by $\F_q$ the field of order $q$. 
We extend this notation to any integer $n>1$ denoting by $\F_n$ the ring which is direct product of all the fields 
whose orders are the components of $n$. Thus, for instance, $\F_{45}=\F_5\times\F_9$. 
The additive group of the ring $\F_n$ will be denoted by $V_n$ and we set $V_n^*:=V_n\setminus\{0\}$.
If $n=1$, then $V_n$ is the trivial group with one element.
If $d$ is a divisor of $n$, any subgroup $S$ of $V_n$ of order $d$ is clearly isomorphic to $V_d$ and therefore, by abuse of notation,
such a subgroup $S$ will be often denoted by $V_d$.

The group of units of $\F_n$ will be denoted by $\U(\F_n)$ and its order by $\psi(n)$. Obviously, in the particular case that $n=q$ is a prime power,
$\U(\F_n)$ is nothing but the multiplicative group $\F_q^*$ of the field $\F_q$ and $\psi(n)=q-1$. Otherwise, if $n$ has more than one component,
say $q_1, \dots, q_\omega$, then $\U(\F_n)=\F_{q_1}^* \times \dots \times \F_{q_\omega}^*$ and $\psi(n)=\prod_{i=1}^\omega(q_i-1)$. 
The set of non-zero squares and of non-squares of the field $\F_q$ will be denoted by $\F_q^\Box$ and $\F_q^{\not\Box}$, respectively.

If $A$, $B$ are non-empty subsets of $\F_n$, then $AB$ will denote the multiset $\{ab \ | \ a\in A; b\in B\}$. If $A=\{a\}$ or $B=\{b\}$,
then $AB$ will be written as $aB$ or $Ab$, respectively.

Let $q_1$, \dots, $q_\omega$ be the components of an odd integer $n$. For
every non-empty $I$ belonging to the power-set $2^{\{1,\dots,\omega\}}$, choose an element $c(I)\in I$ and
consider the subset $S(I)$ of $\F_n^*$ defined as follows:
$$S(I)=S_1(I) \times \dots \times S_\omega(I) \quad{\rm with}\quad S_j(I)=\begin{cases}\{0\} & \mbox{if $j\not\in I$};
\medskip\cr \F_{q_j}^\Box & \mbox{if $j=c(I)$};
\medskip\cr \F_{q_j}^* & \mbox{ if $j\in I\setminus\{c(I)\}$.}\end{cases}$$
Then define $S:=\displaystyle\bigcup_{I\in 2^{\{1,\dots,\omega\}}\setminus\{\emptyset\}} S(I)$.
Such a set $S$ has size ${n-1\over2}$ and then will be called  a {\it halving} of $\F_n^*$.
It is easy to see that it has the following property:
\begin{equation}\label{halving}
x, y\in \F_n \ \ {\rm and} \ \ x_iy_i\in \F_{q_i}^{\not\Box} \ {\rm for} \ 1\leq i\leq \omega \ \Longrightarrow\{x,y\}S=\F_n^*.
\end{equation}

Given a group $G$ and an element $s$ of $\F_n$, the endomorphism
of $G\times V_n$ mapping $(x,y)$ to $(x,ys)$ will be denoted by $\mu_s$.
It is evident that if $s\in \U(\F_n)$, then $\mu_s$ is an automorphism of $G\times V_n$.

Given $\alpha\geq1$, remember that throughout the paper $G_\alpha$ will denote the group
defined in the ``if part" of Theorem \ref{PertinentOrders}.
By {\it canonical involution} of $G_\alpha$ we will mean $(0,2^{\alpha-1},2^{\alpha-1})$.
Also, the canonical involution of $G_\alpha\times V_n$ will be $(0,2^{\alpha-1},2^{\alpha-1},0)$.
Speaking of a $(G,H,3,1)$-RDF with $G=G_\alpha$ or $G=G_\alpha\times V_n$,
we will mean a $\{0,j\}$-resolvable $(G_\alpha,H,3,1)$-DF where $j$ is
the canonical involution of $G$.

\section{The smallest examples}
The five smallest {\it pertinent} values of $n$ are 6, 12, 30, 36 and 48.
In the following, for each of these values, a 3-pyramidal KTS$(n+3)$ will be given by means of a $(G,\{2^3,3\},3,1)$-RDF 
with $G=D$, $G_1$, $D\times V_5$, $G_1\times V_3$ and $G_2$, respectively.
By way of illustration, in the first two cases we follow the instructions of Theorem \ref{DF} and we concretely construct a 3-pyramidal KTS(9)
and a 3-pyramidal KTS(15). 

The realizations of these five small KTSs  allow us to state the following.
\begin{prop}\label{matrioska3}
Assume that $\cal F$ is a $(G,H,3,1)$-RDF with $G$ pertinent of order $n$ and $H$ isomorphic to one of the following 
groups: $D$, $G_1$, $D\times V_5$, $G_1\times V_3$ and $G_2$.
Then there exists a $3$-pyramidal KTS$(n+3)$.

Furthermore, if $M$ is a group of $m$ strong multipliers of $\cal F$, 
the obtained KTS$(n+3)$ admits at least $mn$ automorphisms.
\end{prop}
\begin{proof}
The first assertion follows immediately from Proposition \ref{matrioska2}. For the second assertion,
it is enough to observe that the semidirect product $G\rtimes M$ is a group of automophisms of the obtained KTS$(n+3)$.
\end{proof}

The above proposition will enable us to construct infinite classes of 3-pyramidal KTSs in the subsequent sections.

\subsection{A 3-pyramidal KTS$(9)$}
The set of all non-trivial subgroups of $D$ is a spread of type $\{2^3,3\}$.
Thus the empty family can be seen as a $J$-resolvable $(D,\{2^3,3\},3,1)$-DF
with $J$ any of the three subgroups of $D$ of order 2. Applying Theorem \ref{DF}
with $(j_1,j_2,j_3)=((1,0),(1,1),(1,2))$ and $(g_2,g_3)=((1,2),(1,1))$,
we get a 3-pyramidal representation of the unique KTS(9) (that is the affine plane of order 3) with
point-set $D \ \cup \ \{\infty_1,\infty_2,\infty_3\}$ and the following parallel
classes:
$$\begin{matrix}
\{\infty_1,\infty_2,\infty_3\} \ & \ \{(0,0),(0,1),(0,2)\} \ & \ \{(1,0),(1,1),(1,2)\}\cr
\{\infty_1,(0,0),(1,0)\} \ & \  \{\infty_2,(0,1),(1,2)\} \ & \  \{\infty_3,(0,2),(1,1)\}\cr
\{\infty_1,(0,1),(1,1)\} \ & \  \{\infty_2,(0,2),(1,0)\} \ & \  \{\infty_3,(0,0),(1,2)\}\cr
\{\infty_1,(0,2),(1,2)\} \ & \  \{\infty_2,(0,0),(1,1)\} \ & \  \{\infty_3,(0,1),(1,0)\}.\cr
\end{matrix}$$

\subsection{A 3-pyramidal KTS$(15)$}\label{15}
The set of all non-trivial subgroups of $G_1$ is a spread of this group of type $\{2^3,3^4\}$.
Let $\Sigma$ be the $\{2^3,3\}$-partial spread obtained from it by removing all the 3-subgroups except
$H=\{(0,0,0),(1,0,0),(2,0,0)\}$. Consider the 3-subset $B=\{(0,0,1),(1,1,0),(2,1,1)\}$ of $G_1$.
Looking at its ``difference table"
\begin{center}
\begin{tabular}{|c|||c|c|c|c|c|c|c|c|c|c|c|c|c|c}
\hline {$\hat{-}$} & $(0,0,1)$ & $(1,1,0)$ & $(2,1,1)$   \\
\hline\hline
\hline $(0,0,1)$ & $\bullet$ & $\bf(2,0,1)$ & $\bf(1,0,1)$  \\
\hline $(1,1,0)$ & $\bf(1,1,1)$ & $\bullet$ & $\bf(2,1,1)$  \\
\hline $(2,1,1)$ & $\bf(2,1,0)$ & $\bf(1,1,0)$ & $\bullet$  \\
\hline
\end{tabular}\quad\quad\quad
\end{center}
we see that $\Delta B$ is the set of all the 3-elements of $G_1$ not belonging to $H$.
Thus the singleton ${\cal F}=\{B\}$ is a $(G_1,\Sigma,3,1)$-DF.

Now consider the subgroup $J=\{(0,0,0),(0,1,1)\}$ of $G_1$. The other two subgroups of order 2 are
$a+J-a$ and $b+J-b$ with $a=(1,1,1)$ and $b=(2,0,1)$. Now partition $G_1$ into the left cosets of $J$ 
indicating in boldface the elements of $B \ \cup \ \{0,a,b\}$:
$$\{{\bf(0,0,0)},(0,1,1)\},\quad \{{\bf(0,0,1)},(0,1,0)\},\quad \{(1,0,0),{\bf(1,1,1)}\}\},$$
$$\{(1,0,1),{\bf(1,1,0)}\},\quad \{(2,0,0),{\bf(2,1,1)}\},\quad \{{\bf(2,0,1)},(2,1,0)\}.$$
We see that $B \ \cup \ \{0,a,b\}$ is a system of representatives for the left cosets of $J$ in $G_1$, i.e.,
${\cal F}$ is $J$-resolvable. Following the instructions given in the proof of the ``if part" of Theorem \ref{DF},
we obtain the following 3-pyramidal representation of a KTS(15) where, to save space, each element $(a,b,c)\in G_1$
is written as $abc$.
\small
$$\begin{matrix}
\{\infty_1,\infty_2,\infty_3\} & \{000,100,200\} &\{001,101,201\} & \{010,110,210\} & \{011,111,211\}\cr
 \{\infty_1,000,011\} & \{\infty_2,111,100\} & \{\infty_3,201,210\} & \{001, 110, 211\} & \{010, 101, 200\}\cr
 \{\infty_1,001,010\} & \{\infty_2,110,101\} & \{\infty_3,200,211\} & \{000, 111, 210\} & \{011, 100, 201\}\cr
 \{\infty_1,100,110\} & \{\infty_2,210,200\} & \{\infty_3,011,001\} & \{111, 201, 010\}  & \{101, 211, 000\}\cr
 \{\infty_1,101,111\} & \{\infty_2,211,201\} & \{\infty_3,010,000\} & \{110, 200, 011\}  & \{100, 210, 001\}\cr
 \{\infty_1,200,201\} & \{\infty_2,001,000\} & \{\infty_3,110,111\} & \{210, 011, 101\} & \{211, 010, 100\}\cr
\{\infty_1,211,210\} & \{\infty_2,010,011\} & \{\infty_3,101,100\} & \{201, 000, 110\} & \{200, 001, 111\}\cr
\end{matrix}$$
\normalsize
It is known that, up to isomorphism, there exist exactly seven KTS(15), i.e., there are seven non-isomorphic
solutions to the well-known Kirkman fifteen schoolgirls problem. It is possible to show, applying the proposition 
on page 894 of \cite{FP}, that the solution above is necessarily isomorphic to the original solution given by Kirkman,
that is the solution denoted by 1a in \cite{CD} Table 1.28, p. 30.

\subsection{A $3$-pyramidal KTS$(33)$}\label{33}
Let $G=D\times V_5$ and consider the following four triples of $G$:
$$\{(0,0,3),(0,0,2),(0,2,4)\},\quad \{(0,0,1),(0,1,3),(1,2,2)\},$$
$$\{(0,0,4),(0,2,3),(1,1,1)\},\quad \{(0,1,1),(0,1,4),(1,1,2)\}.$$
One can check that they form a $(G,\Sigma,3,1)$-DF where $\Sigma$ is the unique $\{2^3,3\}$-partial spread
of $G$, that is the set of all non-trivial subgroups of $D$.
Then check that this difference family is $J$-resolvable with $J=\{(0,0,0),(1,0,0)\}$; two elements $a$, $b$
as in Definition \ref{JRDF} are $(1,1,0)$ and $(1,2,0)$.

\subsection{A $3$-pyramidal KTS$(39)$}\label{39}
Let $G=G_1\times V_3$ and consider the following five triples of $G$:
$$\{ (0, 0, 0, 2), (0, 1, 1, 1),(1, 0, 0, 1)\},$$
$$\{(0, 0, 1, 2), (2, 0, 1, 2), (2, 1, 0, 1)\},$$
$$\{(0, 1, 0, 0), (1, 0, 1, 2), (2, 0, 1, 0)\},$$
$$\{(0, 1, 0, 1), (1, 1, 1, 0), (2, 0, 0, 0)\},$$
$$\{(1, 0, 1, 1), (1, 1, 0, 0),  (2, 1, 1, 2)\}.$$
One can check that they form a $(G,\Sigma,3,1)$-DF where $\Sigma$ is the $\{2^3,3\}$-partial spread whose
subgroup of order 3 is $\{0,x,-x\}$ with $x=(0,0,0,1)$. Then check that this difference family is $J$-resolvable with 
$J=\{(0,0,0,0),(0,1,1,0)\}$; two elements $a$, $b$
as in Definition \ref{JRDF} are $(1,0,0,2)$ and $(2,0,0,1)$.

\subsection{A $3$-pyramidal KTS$(51)$} \label{51}
One can check that the following six triples of $G_2$:
$$B_1=\{(0,0,1),(2,3,0),(2,3,1)\},\quad B_2=\{(0,1,1),(0,1,2),(0,2,1)\},$$
$$B_3=\{(1,1,0),(1,0,1),(2,1,1)\},\quad B_4=\{(1,1,2),(2,0,3),(1,3,1)\},$$
$$B_5=\{(2,1,0),(1,0,3),(0,3,1)\},\quad B_6=\{(0,1,0),(2,2,3),(1,3,3)\},$$
form a $(G_2,H,3,1)$-RDF where $H$ is the subgroup of $G_2$ with underlying-set $\Z_3\times2Z_4\times2\Z_4$.
The map $(a,b,c)\in G_1 \longrightarrow (a,2b,2c)\in H$ is clearly an isomorphism between $G_1$ and $H$.
Hence the singleton $\{B_7\}$ with $B_7=\{(0,0,2),(1,2,0),(2,2,2)\}$ is a $(H,\{2^3,3\},3,1)$-RDF 
for what we have seen in Subsection \ref{15}. We conclude that $\{B_1,\dots,B_6,B_7\}$ is a $(G_2,\Sigma,3,1)$-RDF 
where $\Sigma$ is the $\{2^3,3\}$-partial spread whose
subgroup of order 3 is $\{0,x,-x\}$ with $x=(1,0,0)$.

\section{Three direct constructions}

The action of a group $U$ on a set $V$ is said to be {\it semiregular} if the non-identity elements of $U$ do not fix any element of $V$.
The following fact is straightforward.
\begin{prop}\label{US}
If $U$ is a group of units of $\F_n$ whose action by multiplication on $V_n^*$ is semiregular and $S$ is a complete system of representatives 
for the orbits of $U$ on $V_n^*$, then we have $US=V_n^*$.
\end{prop}

We need the following lemma.
\begin{lem}\label{semiregular}
Let $n>1$ be an integer whose components are all congruent to $1$ $($mod $\lambda)$. 
Then there exist a unit $u$ of $\F_n$ of order $\lambda$ and a subgroup $T$ of $\U(\F_n)$ such that
\begin{enumerate}
  \item $u^j-1$ is a unit for $1\leq j\leq \lambda-1$ and the group $U$ generated by $u$ acts semiregularly on $V_n^*$;
  \item the order of $T$ is the greatest divisor of $\psi(n)$ coprime with $\lambda$;
  \item $T$ leaves invariant a suitable complete system $S$ of representatives for the orbits of $U$ on $V_n^*$.
\end{enumerate}
\end{lem}
\begin{proof}
Let $q_1$, \dots, $q_\omega$ be the components of $n$.
For $1\leq i\leq \omega$, let $u_i$ be a generator of the subgroup $U_i$ of $\F_{q_i}^*$ of order $\lambda$, and set $u=(u_1,\dots,u_\omega)$.
It is very easy to prove that $u$ satisfies 1. (see Corollary 3.3 and Lemma 3.2 in \cite{disjoint}, in this order).
Let $T_i$ be the subgroup of $\F_{q_i}^*$ whose order is the greatest divisor of $q_i-1$ coprime with 
$\lambda$, set $T = T_1\times \dots \times T_\omega$, and let $\Sigma_i$ be a complete system of representatives for the cosets of $T_iU_i$ in
$\F_{q_i}^*$. Now, for
every non-empty $I$ belonging to the power-set $2^{\{1,\dots,\omega\}}$, choose an element $c(I)\in I$ and
consider the subset $S(I)$ of $\F_n^*$ defined as follows:
$$S(I)=S_1(I) \times \dots \times S_\omega(I) \quad{\rm with}\quad S_j(I)=\begin{cases}\{0\} & \mbox{if $j\not\in I$};
\medskip\cr \Sigma_jT_j & \mbox{if $j=c(I)$};
\medskip\cr \F_{q_j}^* & \mbox{ if $j\in I\setminus\{c(I)\}$.}\end{cases}$$
Then set $S=\displaystyle\bigcup_{I\in 2^{\{1,\dots,\omega\}}\setminus\{\emptyset\}} S(I)$.
It is not difficult to check that  $S$ is a complete system of representatives for the orbits of $U$ on $V_n^*$ and that
$T$ leaves $S$ invariant.
\end{proof}

In this section we give three direct constructions which can be understood without any further explanation.
Anyway, to be more informative, we emphasize that, in each case, a suitable {\it strong difference family}
satisfying a special resolvability property has been used. These concepts are defined as follows.

\begin{defn}
Given a group $G$ and an even integer $\lambda$,
a $(G,3,\lambda)$ {\it strong difference family} (SDF) is a collection of triples of elements of $G$
whose list of differences covers each element of $G$,  $0$ included, exactly $\lambda$ times.

If $J$ is a subgroup of $G$ of order $2$, then we say that a $(G,3,\lambda)$-SDF 
is $J$-resolvable if its flatten contains exactly $\lambda$ elements of each left coset of $J$ in $G$.
\end{defn}

Although the notion of a SDF was implicitly used in the literature for a long time, the formal definition has been 
given in \cite{B99}. Since then, SDFs have been crucial for the construction of various combinatorial designs 
in several papers such as  \cite{BuGio,BP,CCFW,CFW1,CFW2,Momihara,YYL}.
As far as we are aware, the notion of a $J$-resolvable SDF is new.

\begin{thm}\label{9mod24}
If all the components of $4n+1$ are congruent to $1$ mod $4$, then there exists a
$(D \times V_{4n+1},D \times V_1,3,1)$-RDF with a group of strong multipliers whose order 
is the greatest odd divisor of $\psi(4n+1)$.
\end{thm}
\begin{proof}
Take $u$, $T$ and $S$ as in the statement of Lemma \ref{semiregular} applied with $\lambda=4$.
Thus $u$ is a unit of order 4 such that $u-1$ is a also a unit and $U:=\langle u\rangle$ acts semiregularly on $V_{4n+1}^*$. 
Note that we necessarily have $u^2=-1$, hence $U=\{\pm1,\pm u\}$. Also note that we have $(u-1)U=\{\pm(u-1),\pm(u+1)\}$.

Let us consider the set ${\cal B}$ consisting of the following triples of $D\times V_{4n+1}$ (recall that the underlying set of $D$ is $\Z_2\times\Z_3$; see Section \ref{DFs}):
$$\{(0,0,u),(0,0,-u),(0,2,-1)\},\quad \{(0,0,1),(0,1,u),(1,2,-u)\},$$
$$\{(0,0,-1),(0,2,u),(1,1,1)\},\quad \{(0,1,1),(0,1,-1),(1,1,-u)\}.$$

\medskip
It is straightforward to check that we have
$\Delta{\cal B}=\bigcup_{g\in D}\{g\}\times\Delta_g$
with $$\Delta_g=\begin{cases}2U & \mbox{ if $g=(0,0)$ or $g=(1,1)$};
\medskip\cr (u-1)U & \mbox{otherwise}\end{cases}$$ 
It is also readily seen that  $$\bigcup_{B\in{\cal B}} B+J=\bigcup_{g\in D}\{g\}\times\Phi_g$$
with $J=\{(0,0,0),(1,0,0)\}$ and $\Phi_g=U$ for every $g\in D$.
Now set
$${\cal F}=\{\mu_s(B) \ | \ s\in S; B\in{\cal B}\}.$$ 
Given that $S$ is a complete system of representatives for the orbits of $U$ on $V_{4n+1}^*$,
we have $US=V_{4n+1}^*$ by Proposition \ref{US} and then, taking into account the previous identities, we easily obtain
$$\Delta{\cal F}=(D\times V_{4n+1})\setminus(D\times V_1)=\Phi({\cal F})+J$$
which means that ${\cal F}$ is a $(D\times V_{4n+1}, D\times V_1,3,1)$-RDF.

Finally, given that $T$ is a subgroup of $\U(\F_n)$ which leaves $S$ invariant, we infer that $\{\mu_t \ | \ t\in T\}$ is a group of strong multipliers of $\cal F$. 
The assertion follows by observing that this group is clearly isomorphic to $T$ and reminding that the order of
$T$ is the greatest odd divisor of $\psi(4n+1)$.
\end{proof}

Observe that the set of {\it initial} base blocks ${\cal B}$ considered in the above theorem
is a lifting of a $J$-resolvable $(D,3,4)$-SDF.

If we apply Theorem \ref{9mod24} with $n=1$ we are forced to take $u=3$ and one can see that the resultant RDF is exactly the one given in Subsection \ref{33}.

\begin{thm}\label{15mod24}
If all the components  of $n$ are congruent to $1$ mod $4$, then there exists a
$(G_1\times V_n,G_1\times V_1,3,1)$-RDF with a group of strong 
multipliers whose order 
is the greatest odd divisor of $\psi(n)$.
\end{thm}

\begin{proof}
Again, as in Theorem \ref{9mod24}, take $u$, $T$ and $S$ as in the statement of Lemma \ref{semiregular} applied with $\lambda=4$.
Consider the set ${\cal B}$ consisting of the following triples of $G_1\times V_n$
(recall that the underlying set of $G_1$
is $\Z_3\times\Z_2\times\Z_2$; see Section \ref{DFs}):
$$\{(0, 0, 0, -1), (0, 0, 0, 1), (2, 1, 0, -u)\},\quad\quad 
\{(0, 0, 0, -u), (0, 0, 0, u), (2, 1, 1, 1)\},$$ 
$$\{(0, 0, 1,  1), (0, 1, 0, -1), (1, 1, 0, -u)\},\quad\quad 
\{(0, 0,  1, u), (0, 1, 0, -u), (1, 1, 0, 1)\},$$ 
$$\{(1, 0,   0, -1), (1, 1, 1, 1), (2, 0, 1, u)\},\quad\quad
\{(1, 0,   0, u), (1, 1, 1, -u), (2, 1, 1, -1)\},$$ 
$$\{(1, 0,   1, -1), (2, 0, 0, -u), (2, 1, 1, u)\},\quad\quad
\{(1, 0,   1, u), (2, 0, 1, -1), (2, 1, 0, 1)\}.$$

\medskip
One can check that 
$\Delta{\cal B}=\bigcup_{g\in G_1}\{g\}\times\Delta_g$
with $\Delta_g=2U$ or $\Delta_g=(u-1)U$ according to whether 
$g$ belongs or does not belong to the Klein subgroup of $G_1$, respectively. 

Also, we have $\bigcup_{B\in{\cal B}} B+J=\bigcup_{g\in G_1}\{g\}\times\Phi_g$
with $J=\{(0,0,0,0),(0,1,1,0)\}$ and $\Phi_g=U$ for every $g\in G_1$.

Reasoning as in Theorem \ref{9mod24}, we can see that $${\cal F}=\{\mu_s(B) \ | \ s\in S; B\in{\cal B}\}$$
is a $(G_1\times V_n,G_1\times V_1,3,1)$-RDF admitting $\{\mu_t \ | \ t\in T\}$ as
a group of strong multipliers of order the greatest odd divisor of $\psi(n)$.
\end{proof}

Analogously to Theorem \ref{9mod24} the set of {\it initial} base blocks ${\cal B}$ considered above 
is a lifting of a $J$-resolvable $(G_1,3,4)$-SDF.

\begin{thm}\label{15mod24bis}
If all the components of $n$ are congruent to $1$ mod $6$, then there exists a
$(G_1\times V_n,G_1\times V_1,3,1)$-RDF with a group of strong multipliers whose order 
is the greatest divisor of $\psi(n)$ coprime with $6$.
\end{thm}

\begin{proof}
Take $u$, $T$ and $S$ as in the statement of Lemma \ref{semiregular} applied with $\lambda=6$.
Thus $u$ is a unit of $\U(\F_n)$ of order 6 such that $u^i-1$ is a unit 
for $1\leq i\leq 5$ and $U:=\langle u\rangle$ acts semiregularly on $V_n^*$.
Note that we necessarily have $u^3=-1$ so that $U=\{\pm1,\pm u, \pm u^2\}$, and the identity $u^2-u+1=0$ holds
\footnote{By definition, we have $u^6-1=0$, hence $(u^3-1)(u+1)(u^2-u+1)=0$. It cannot be $u^3-1=0$ or $u+1=0$ otherwise
$u$ would have order 3 or 2, respectively. It follows that $u^2-u+1=0$.}.
Consider the set ${\cal B}$ consisting of the following triples of $G_1\times V_n$:
$$\{(0, 0, 0, 1), (0, 0, 0, -u), (0, 0, 0, u^2)\};$$
$$\{(0, 0, 0, u), (0, 1, 0, -u^2), (1, 1,  0, -1)\};$$
$$\{(0,  0, 1, -u), (1, 0, 0, u^2), (2, 0, 1, 1)\};$$
$$\{(0,  0, 1, u^2), (1, 0, 1, 1), (1, 1, 1, -u)\};$$
$$\{(0,  0, 1, 1), (1, 1, 0, u^2), (2, 0,   0, -u)\};$$
$$\{(0, 1, 0, u), (0, 1, 1, -u^2), (2, 0, 1, -1)\};$$
$$\{(0, 1, 0, -1), (1, 1, 1, u), (2, 0,   0, -u^2)\};$$ 
$$\{(0, 1, 1, -1), (1, 0,  0, -u^2), (2, 0, 1, u)\};$$
$$\{(1, 0, 0, 1), (1, 0, 1, -u), (2, 0,   1, u^2)\};$$ 
$$\{(1, 0,  1, -u^2), (1, 1, 1, -1), (2, 1, 1, u)\};$$
$$\{(1, 0, 1, u), (2, 0,   1, -u^2), (2, 1, 1, -1)\};$$
$$\{(2, 0, 0, u^2), (2, 0, 1, -u), (2, 1, 1, 1)\}.$$

\medskip
With a little bit of patience, taking into account the identity $u^2=u-1$, it is not difficult to check that we have: 
\begin{equation}\label{DeltaPhi3}
\Delta{\cal B}=\bigcup_{g\in G_1}\{g\}\times(u+1)U \quad{\rm and}\quad
\bigcup_{B\in{\cal B}} B+J=\bigcup_{g\in G_1}\{g\}\times U
\end{equation}
where $J=\{(0,0,0,0),(0,1,1,0)\}$. We can write $u+1=-(u^4-1)$, hence $u+1$ is a unit of $\F_n$ by assumption on $u$.
Now set
$${\cal F}=\{\mu_s(B) \ | \ s\in S; B\in{\cal B}\}.$$ 
Given that $S$ is a complete system of representatives for the orbits of $U$ on $V_{4n+1}^*$,
we have $US=V_n^*$ by Proposition \ref{US} and then, taking into account (\ref{DeltaPhi3}), we easily obtain
$$\Delta{\cal F}=\Phi({\cal F})+J=(G_1\times V_n)\setminus(G_1\times V_1)$$
which means that ${\cal F}$ is a $(G_1\times V_n,G_1\times V_1,3,1)$-RDF.\\
Finally, given that $T$ is a subgroup of $\U(\F_n)$ which leaves $S$ invariant, we infer that $\{\mu_t \ | \ t\in T\}$ is a group of strong multipliers of $\cal F$. 
The assertion follows observing that this group is clearly isomorphic to $T$ whose order
is the greatest divisor of $\psi(n)$ coprime with 6.
\end{proof}

This time the set of {\it initial} base blocks ${\cal B}$ considered above
is a lifting of a $J$-resolvable $(G_1,3,6)$-SDF.

\section{A composition construction via pseudo-resolvable difference families}\label{prdf}

Now we define a class of $(G,\{2^3,3\},3,1)$-DFs that we call pseudo-resolvable and
will be crucial for the construction of a 3-pyramidal KTS$(v)$ with $v=36n+3$ or $v=48n+3$ or $v=108n+3$ with all the 
components of $n$ congruent to $7$ or 11 (mod 12).

\begin{defn}
Let $G$ be a pertinent group of doubly even order and let $\cal F$ be a $(G,\Sigma,3,1)$-DF with
$\Sigma=\{\{0,j_1\},\{0,j_2\},\{0,j_3\},\{0,x,-x\}\}$.  We say that $\cal F$ is {\it pseudo-resolvable} (PRDF for short) if $\Phi({\cal F}) \ \cup \ \{0,j_\alpha,x\}$
is a complete system of representatives for the left cosets of $\{0,j_\beta\}$ in $G$ for suitable involutions $j_\alpha$, $j_\beta$.
\end{defn}

Even though the definitions of resolvable and pseudo-resolvable difference families are very similar, they
are independent. For instance, in spite of the fact that there exists a $(G_1,\{2^3,3\},3,1)$-RDF (see Subsection \ref{15}),
there is no $(G_1,\{2^3,3\},3,1)$-PRDF.
Certainly, a $(G,\{2^3,3\},3,1)$-DF cannot be resolvable and pseudo-resolvable at the same time.
Here is a useful example in the groups $G_1\times V_3$.

\begin{ex}\label{36PRDF}
Let $\Sigma$ be the $\{2^3,3\}$-partial spread of $G_1\times V_3$ 
whose member of order $3$ is $\{0,x,-x\}$ with $x=(1,0,0,0)$.
One can check that the following five triples of $G_1\times V_3$ 
$$B_1=\{(0, 0, 1, 2), \ (0, 1, 1, 1), \ (1, 0, 0, 2)\},$$
$$B_2=\{(0, 1, 0, 1),\  (1, 0, 0, 1), \ (2, 1, 0, 0)\},$$
$$B_3=\{(0, 1, 1, 2), \ (2, 0, 0, 0), \ (2, 0, 0, 1)\},$$
$$B_4=\{(1, 0, 1, 1), \  (1, 1, 0, 0), \ (2, 1, 0, 2)\},$$
$$B_5=\{(1, 1, 0, 2), \ (2, 0, 0, 2), \ (2, 0, 1, 1)\},$$
form a $(G_1\times V_3,\Sigma,3,1)$-PRDF.
\end{ex}

The following result explains why pseudo-resolvable DFs can be helpful in the
construction of some resolvable DFs.

\begin{thm}\label{PRDF}
Assume that all the components of $n$ are congruent to $3$ mod $4$ but distinct from $3$,
and that there exists a $(G,\Sigma,3,1)$-PRDF with $G$ pertinent of doubly even order. Then there exists 
a $(G\times V_n,G\times V_1,3,1)$-RDF with a group of 
strong multipliers of order the greatest odd divisor of $\psi(n)$.
\end{thm}
\begin{proof}
The fact that $G$ is pertinent of doubly even order implies that its three involutions $j_1$, $j_2$, $j_3$, together with zero,
is the Klein four-group (see Lemma \ref{klein}). Thus we have $j_\alpha+j_\beta=j_\gamma$ for every permutation 
$(\alpha,\beta,\gamma)$ of $(1,2,3)$.

Let $\cal F$ be a $(G,\Sigma,3,1)$-PRDF and let $H$ be the union of the
members of $\Sigma$. Thus $H=\{0,j_1,j_2,j_3,x,-x\}$ for a suitable element $x$ of order 3. By definition, up to a reordering of $\{j_1,j_2,j_3\}$, we have
\begin{equation}\label{PRDFB}
\Delta{\cal F}=G\setminus H\quad\quad\mbox{and}\quad\quad \Phi({\cal F})+\{0,j_1\}=G\setminus L
\end{equation}
where $L=\{0,j_1,j_2,j_3,x,x+j_1\}$.

Let $n=q_1\dots q_\omega$ be the prime power factorization of $n$. By assumption we have $q_i\equiv3$ (mod 4), hence $-1\in\F_{q_i}^{\not\Box}$ for $1\leq i\leq \omega$.
Take any element $\sigma_i$ of $\F_{q_i}^\Box\setminus\{1\}$ and set $y_i={\sigma_i+1\over \sigma_i-1}$. Set $y=(y_1,\dots,y_\omega)$ and consider the
following  two triples of $G\times\F_n$:
$$A_1=\{(0,1), \ (x, y), \ (x,-y)\},\quad A_2=\{(0,-1), \ (j_2,y), \ (j_3,-y)\}.$$
Note that we have $\displaystyle\Delta\{A_1,A_2\}=\bigcup_{h\in H}\{h\}\times \Delta_h$ with 
$$\Delta_{0}=\Delta_{j_1}=\{2y,-2y\};$$
$$\Delta_{j_2}=\{y+1,-(y+1)\};\quad \Delta_{j_3}=\{y-1,-(y-1)\};$$
$$\Delta_x=\{y-1,-(y+1)\};\quad \Delta_{-x}=\{y+1,-(y-1)\}.$$

We have ${y_i+1\over y_i-1}=\sigma_i\in\F_{q_i}^\Box$,
thus $y_i-1$ and $y_i+1$ are both squares or both non-squares of $\F_{q_i}$. 
Also, we have $q_i\equiv3$ (mod 4) for each $i$ so that $-1\in \F_{q_i}^{\not\Box}$. 
On the basis of these facts it is evident that the projection of each $\Delta_h$ on $\F_{q_i}$ consists of a 
non-zero square and a non-square of $\F_{q_i}$. As a consequence, if $S$ is any halving of $\F_n$, we have $S\Delta_h=\F_n^*$ for every $h\in H$. Thus,  
the set of triples ${\cal A}=:\{\mu_s(A_i) \ | \ i=1,2; s\in S\}$ has list of differences $\Delta{\cal A}=H\times V_n^*$, i.e.,
\begin{equation}\label{DeltaA}
\Delta{\cal A}=(H\times V_n)\setminus(H\times V_1)
\end{equation}
Now note that we have 
$$(A_1 \ \cup \ A_2)+\{(0,0),(j_1,0)\}=\bigcup_{\ell\in L}\{\ell\}\times\{\phi_\ell,-\phi_\ell\}$$
with $\phi_0=\phi_{j_1}=1$, and $\phi_\ell=y$ for $\ell\in L\setminus\{0,j_1\}$. We clearly have
$\{\phi_\ell,-\phi_\ell\}\cdot S=V_n^*$ for each $\ell\in L$ and then
\begin{equation}\label{cosets1}
\Phi({\cal A})+\{(0,0),(j_1,0)\}=(L\times V_n)\setminus(L\times V_1).
\end{equation}
Take a triple $\{u_1,u_2,u_3\}$ of units of $\F_n$ with the property that the elements of its list of differences 
$\Delta\{u_1,u_2,u_3\}$ are also units. For instance, one could take $\{u_1,u_2,u_3\}=\{1,-1,y\}$.
Now lift each triple $B=\{b_1,b_2,b_3\}\in{\cal F}$ to the triple $B^+=\{(b_1,u_1),(b_2,u_2),(b_3,u_3)\}$
and set $${\cal F}^+=\{\mu_z(B^+) \ | \ B\in{\cal F}; z\in V_n^*\}.$$ 
We note that the contribution of a single $B\in{\cal F}$ to $\Delta{\cal F}^+$ and $\Phi({\cal F}^+)$
is $(\Delta B)\times V_n^*$ and $B\times V_n^*$, respectively. Thus that the two equalities in (\ref{PRDFB}) imply that we have:
\begin{equation}\label{DeltaF+}
\Delta{\cal F}^+=(G\setminus H)\times(V_n\setminus V_1);
\end{equation}
\begin{equation}\label{cosets2}
\Phi({\cal F}^+) + \{(0,0),(j_1,0)\}=(G\setminus L)\times(V_n\setminus V_1).
\end{equation}
It is clear that ${\cal A} \ \cup \ {\cal F}^+$ is a $(G\times V_n,G\times V_1,3,1)$-RDF.  
Indeed (\ref{DeltaA}) and (\ref{DeltaF+}) imply that $\Delta({\cal A} \ \cup \ {\cal F}^+)=(G\times V_n)\setminus(G\times V_1)$.
Also, (\ref{cosets1}) and (\ref{cosets2}) imply that 
$\Phi({\cal A} \ \cup \ {\cal F}^+) + \{(0,0),(j_1,0)\}=(G\times V_n)\setminus(G\times V_1)$.

Letting $q_1, \dots, q_\omega$ be the components of $n$, we finally note that 
$$\{\mu_s \ | \ s\in \F_{q_1}^\Box \times \dots \times \F_{q_\omega}^\Box\}$$ 
is a group of strong multipliers of ${\cal F}^+$
and its order is the greatest odd divisor of $\psi(n)$ since each $q_i\equiv3$ (mod 4). The assertion follows.
\end{proof}

Letting $G=G_1\times V_3$ and applying Theorem \ref{PRDF} to the $(G,\Sigma,3,1)$-PRDF given
in Example \ref{36PRDF}, we obtain the following important result.

\begin{cor}\label{36P}
If all the components of $n$ are greater than $3$ and congruent to $3$ modulo $4$, then there exists a $(G_1\times V_{3n},G_1\times V_3,3,1)$-RDF
with a group of strong multipliers of order the greatest odd divisor of $\psi(n)$.
\end{cor}

In Appendix \ref{G1xV9,3,1PRDF} and \ref{G2,3,1PRDF}  we will give an example of a $(G,\{2^3,3\},3,1)$-PRDF with both $G=G_1\times V_9$ and $G=G_2$. 
Thus, as another important application of Theorem \ref{PRDF} we get the following.

\begin{cor}\label{12*9P}
If all the components of $n$ are greater than $3$ and congruent to $3$ modulo $4$, then there exists a $(G_1\times V_{9n},G_1\times V_9,3,1)$-RDF
with a group of strong multipliers of order the greatest odd divisor of $\psi(n)$.
\end{cor}

\begin{cor}\label{48P}
If all the components of $n$ are greater than $3$ and congruent to $3$ modulo $4$, then there exists a $(G_2\times V_{n},G_2\times V_1,3,1)$-RDF
with a group of strong multipliers of order  the greatest odd divisor of $\psi(n)$.
\end{cor}

\section{Doubly disjoint difference families}

A $(G,H,3,1)$-DF is said to be {\it disjoint} if its blocks are pairwise disjoint and do not meet the subgroup $H$.
It is well known that there exists a disjoint $(\Z_{6n+1},3,1)$-DF and a disjoint $(\Z_{6n+3},\{0,2n+1,4n+2\},3,1)$-DF 
for every positive integer $n$ (see \cite{BuGhi, DR,DS}).

Let us say that two difference families are {\it strongly equivalent} if each block of the first is a suitable translate of a block of 
the second\footnote{We cannot simply say that they are {\it equivalent} since two difference families in a group $G$,
say ${\cal F}=\{B_1,\dots,B_n\}$ and ${\cal F}'=\{B'_1,\dots,B'_n\}$, are usually said to be equivalent if, up to the order, 
we have $B'_i=\alpha(B_i)+t_i$ for a suitable $\alpha\in Aut(G)$ and suitable $t_1, \dots, t_n\in G$.}. 

\begin{defn}
A $(G,H,3,1)$-DF is {\it doubly disjoint} if its blocks tile $G\setminus H$ together with the blocks of another
DF strongly equivalent to it.
\end{defn}

Note how the previous definition reminds a bit of the notion introduced in \cite{CKZ} of a $(v,k,\lambda)$ 
tiling of a group $G$, that is a set of mutually disjoint $(v,k,\lambda)$ difference sets in $G$ partitioning $G\setminus\{0\}$.
Needless to say, however, that the members of a $(v,k,\lambda)$ tiling of $G$ are pairwise strongly inequivalent.

Note that any $\{0,j\}$-resolvable $(G,H,3,1)$-DF, say ${\cal F}=\{B_1,\dots,B_n\}$, is doubly disjoint. Indeed, 
setting $B'_i=B_i+j$ for $i=1,\dots, n$, it is clear that ${\cal F}$ and ${\cal F}':=\{B'_1,\dots,B'_n\}$ are  strongly equivalent.
Also, by definition, we have $\Phi({\cal F})+\{0,j\}=G\setminus H$ and this is equivalent to say that the blocks of ${\cal F}$
and ${\cal F}'$ form a partition of $G\setminus H$.

In the next section we will give a composition construction for RDFs where doubly disjoint difference families play a crucial 
role. For this reason, it is worth studying their possible existence. In particular, we are interested in 
$(\Z_3\times V_{2n+1},\Z_3\times V_1,3,1)$-DFs. We are going to prove their existence in the case that all the components 
of $2n+1$ are congruent to 1 (mod 4).

\begin{thm}\label{DDDF}
If the components of $n$ are all congruent to $1$ $($mod $4)$, then there exists a doubly disjoint $(\Z_3\times V_n,\Z_3\times V_1,3,1)$-DF.
\end{thm}
\begin{proof}
First observe that if $q\equiv1$ (mod 4) is a prime power, then the set $X=\{x\in \F_q^{\not\Box}  : x-2\in \F_q^\Box\}$ is not empty. 
For instance, using the {\it cyclotomic numbers} of order 2 (see, e.g., \cite{D}) one can see that $X$ has size ${q-1\over4}$.
More elementarily and constructively, an element of $X$ can be found as follows. Take any $y$ of $\F_q^{\not\Box}\setminus\{2\}$ 
and check that the set $$X'=\{y, \ y+1, \ 1-y, \ 4-2y, \ {2\over y+1}, \ {2y\over y-1}\}$$ 
has at least one element in common with $X$.

Thus, if $n=q_1\dots q_\omega$ with $q_i\equiv1$ (mod 4) is a prime power for $1\leq i\leq \omega$, we can construct
an element $x=(x_1,\dots,x_\omega)\in\F_n$ with the property that $x_i\in \F_{q_i}^{\not\Box}$  and $x_i-2\in \F_{q_i}^\Box$ for $1\leq i\leq \omega$.
Consider the 3-subsets $A$ and $B$ of $\Z_3\times V_n$ defined as follows:
$$A=\{(0,1),(1,x),(1,2-x)\};\quad\quad B=\{(0,x),(2,x^2),(2,2x-x^2)\}.$$
We have $$\Delta A \ \cup \ \Delta B=\bigcup_{h=0}^2\{h\}\times \{1,-1\}\cdot\Delta_h$$ 
with $\Delta_0=\{2(x-1),2x(x-1)\}$ and $\Delta_1=\Delta_2={1\over2}\Delta_0$. 
Also, we have $$A \ \cup \ B=\bigcup_{h=0}^2\{h\}\times\Phi_h$$ with $\Phi_0=\{1,x\}$, $\Phi_1=\{x,2-x\}$, and $\Phi_2=\{x^2,x(2-x)\}$.

Now take a halving $S$ of $\F_n^*$ (see Section \ref{notation}).
In view of the choice of the element $x$, we see that the projection of each $\Delta_h$ and each $\Phi_h$ on $\F_{q_i}$ consists of a square and a non-square.
Thus, by (\ref{halving}), we have $\Delta_hS=\Phi_hS=\F_n^*$ for $h=0,1,2$. 

For each $s\in S$ set $A_s=\mu_s(A)$, $B_s=\mu_s(B)$, and consider the family ${\cal F}=\{A_s, B_s \ | \ s\in S\}$. 
We obviously have $\Delta A_s=\mu_s(\Delta A)$ and $\Delta B_s=\mu_s(\Delta B)$. Thus we have 
$$\Delta{\cal F}=\bigcup_{h=0}^2\{h\}\times(\{1,-1\}\cdot\Delta_h\cdot S)=\bigcup_{h=0}^2\{h\}\times(\{1,-1\}\cdot V_n^*)$$
that is two times $(\Z_3\times V_n)\setminus(\Z_3\times V_1)$. 

We also have:
$$\Phi({\cal F})=\bigcup_{h=0}^2\{h\}\times(\Phi_h\cdot S)=\bigcup_{h=0}^2\{h\}\times V_n^*=(\Z_3\times V_n)\setminus(\Z_3\times V_1).$$

Note that the chosen halving $S$ is ${\it symmetric}$, i.e., we have $-S=S$. Then there exists a subset $T$ of $S$ for which
we have $S = T \ \cup \ (-T)$ so that ${\cal F}$ is splittable in the two families
$${\cal F}^+=\{A_t, B_t \ | \ t\in T\},\quad {\cal F}^-=\{A_{-t}, B_{-t} \ | \ t\in T\}.$$ 
Now note that $A_{-t}$ is a translate of $A_t$ and that $B_{-t}$ is a translate of $B_t$ for every $t\in T$. Indeed it is readily seen that we have:
$$A_t+(0,-2t)=A_{-t};\quad\quad B_t+(0,-2xt)=B_{-t}.$$
We deduce, in particular, that $\Delta A_t=\Delta A_{-t}$ and $\Delta B_t=\Delta B_{-t}$ for every $t\in T$. 
It follows that $\Delta{\cal F}^+=\Delta{\cal F}^-$. Thus, considering that $\Delta{\cal F}$ is twice
$\Z_3\times V_n^*$, we necessarily have $\Delta{\cal F}^+=\Delta{\cal F}^-=\Z_3\times V_n^*$.
This means that both ${\cal F}^+$ and ${\cal F}^-$ are $(\Z_3\times V_n,\Z_3\times V_1,3,1)$-DFs.

Considering that each block of ${\cal F}^-$ is a translate of a block of ${\cal F}^+$ and that 
$\Phi({\cal F}^+ \ \cup \ {\cal F}^-)=\Phi({\cal F})=(\Z_3\times V_n)\setminus(\Z_3\times V_1)$,
we conclude that ${\cal F}^+$ is doubly disjoint and the assertion follows.
\end{proof}

\section{Composition constructions via difference matrices}

We recall that a $(h,k,1)$ {\it difference matrix} in a group $H$ of order $h$, briefly denoted by $(H,k,1)$-DM,
is a $k\times h$ matrix in which the difference of any two distinct rows is a permutation of $H$. 

In particular, an $(H,3,1)$-DM is completely equivalent to a {\it complete mapping} of $H$. There is a large
literature on complete mappings starting with the famous conjecture of Hall and Paige \cite{HP} according to which 
a group $H$ of even order admits a complete mapping if and only if it is {\it admissible}, i.e., 
if and only if its 2-Sylow subgroups are not cyclic. The conjecture has been finally proved in \cite{Evans}.

\begin{defn} A $(H,k,1)$ difference matrix is {\it homogeneous} if each row is also a permutation of $H$. \end{defn}

It is quite evident that there exists a homogeneous $(H,k,1)$-DM if and only if there exists
an $(H,k+1,1)$-DM. Indeed, adding a null-row to a homogeneous $(H,k,1)$-DM one gets an $(H,k+1,1)$-DM.
Conversely, if $M$ is a $(H,k+1,1)$-DM with rows $M_1$, \dots , $M_{k+1}$, then one gets a homogeneous $(H,k,1)$-DM 
whose rows are $M_2-M_1$, \dots, $M_{k+1}-M_1$.

Thus, as immediate consequence of the characterization of all abelian groups $H$ for which there exists 
an $(H,4,1)$-DM (see \cite{Ge,PanChang}), we have the following. 

\begin{thm}\label{HDM}
There exists a homogeneous $(H,3,1)$-DM with $H$ abelian of odd order if and only if $|H|>3$.
There exists a homogeneous $(H,3,1)$-DM with $H$ abelian of even order if and only if 
the $2$-Sylow subgroups of $H$ are not cyclic.
\end{thm}

\begin{cor}
There exists a homogeneous $(V_n,3,1)$-DM if and only if $3<n\not\equiv2$ $($mod $4)$.
\end{cor}

Difference matrices are also a crucial topic of Design Theory \cite{BJL,CD}. They have been used explicitly or implicitly in a lot of papers
especially for the composition constructions of designs with a regular automorphism group starting from some early work by
Jungnickel \cite{Jung} and Colbourn \cite{CC}.
Homogeneous difference matrices have been used later for the composition constructions of several kinds of resolvable designs
(see, e.g., \cite{AFL}).
As far as we are aware the quite appropriate term {\it homogeneous} was coined in \cite{KM} when other
authors choose other terms as {\it good} \cite{BZ} about the same time.

We need to devise a new type of difference matrix that we call {\it splittable}.
\begin{defn}
Let $J$ be a subgroup of order $2$ of a group $H$ and let $M$ be a $(H,3,1)$-DM.
We say that $M$ is $J$-splittable if the first half and the second half of each row
of $M$ is a complete system of representatives for the left cosets of $J$ in $H$.
\end{defn}

We give three examples of splittable difference matrices that will be crucial for
the construction of some classes of resolvable difference families.

\begin{ex}\label{(12,3,1)-DM}
Consider the following matrix $M$ with elements in $G_1$
$$\left( \begin{array}{cccccc|cccccc}
000 & 010 & 100 & 110 & 200 & 210 & 011 & 001 & 111& 101&211&201\\
000 &100 & 210 & 010 & 110 & 200 & 000 &101 & 010 & 210 & 200 & 100\\
010&200&210&100&000&110& 101&001&200&011&201&111\\
\end{array}\right)$$
where, to save space, each element $(a,b,c)$ has been denoted by $abc$.
It is straightforward to check that $M$ is a J-splittable $(G_1,3,1)$-DM
with $J=\{(0,0,0),(0,1,1)\}$.
\end{ex}

\begin{ex}\label{(12,3,1)-DMbis}
Consider the following matrix $M$ with elements in $\Z_2\times\Z_6$ 
$$\left( \begin{array}{cccccc|cccccc}
00 & 01 & 02 & 03 & 04 & 05 & 10 & 11 & 12& 13& 14 & 15\\
00 &12 & 15 & 04 & 01 & 03 & 00 &13 & 11 & 05 & 02 & 04\\
03 & 01 & 15 &12 & 00 &04& 11& 05 & 04 & 03 & 12 & 00\\
\end{array}\right)$$
where, to save space, each element $(a,b)$ has been denoted by $ab$.
It is straightforward to check that $M$ is a J-splittable $(\Z_2\times\Z_6,3,1)$-DM
with $J=\{(0,0),(1,0)\}$.
\end{ex}

\begin{ex}\label{(16,3,1)-DM}
Consider the following matrix $M$ with elements in $\Z_4\times\Z_4$ 
$$\left(\begin{array}{cccccccc|cccccccc}
00&30&11&20&01&10&31&21&22&12&33&02&23&32&13&03\\
22&11&30&10&21&23&02&31&13&20&32&00&12&33&01&03\\
22&31&03&01&10&30&20&33&32&12&00&21&13&23&11&02\\
\end{array}
\right)$$
where, to save space, each element $(a,b)$ has been denoted by $ab$.
It is straightforward to check that $M$ is a J-splittable $(\Z_4\times\Z_4,3,1)$-DM
with $J=\{(0,0),(2,2)\}$.
\end{ex}

The matrix of the third example is also homogeneous.
The following theorem explains how doubly disjoint or resolvable difference families in a quotient group $G/H$ can be combined 
with homogeneous or splittable difference matrices in $H$ for the construction of resolvable difference families in $G$.

\begin{thm}\label{DF+DM}
Let $H$ be a normal subgroup of a pertinent group $G$ and let $L$ be a subgroup of $G$ containing $H$.
If ${\cal F}$ is a $(G/H,L/H,3,1)$-DF and $M$ is a $(H,3,1)$-DM, then there exists a $(G,L,3,1)$-DF.

Let $j$ be an involution of $G$ and assume that one of the following additional hypotheses holds.
\begin{itemize}
\item[\rm(i)] $j\not\in H$, $\cal F$ is $\{H,j+H\}$-resolvable, and $M$ is homogeneous;
\item[\rm(ii)] $j\in H$, ${\cal F}$ is doubly disjoint, and $M$ is $\{0,j\}$-splittable.
\end{itemize} 
Then there exists a $\{0,j\}$-resolvable $(G,L,3,1)$-DF.
\end{thm}
\begin{proof}
The first part of the statement has been already proved in \cite{recursive} (see Corollary 5.8). It is convenient, however, to recall how the
$(G,L,3,1)$-DF can be constructed.
Let $^-: g\in G \longrightarrow \overline{g}=g+H\in G/H$, be the canonical epimorphism from $G$ to $G/H$,
let ${\cal F}=\{\overline{B_i} \ | \ i\in I\}$ with $B_i=\{b_{i,1},b_{i,2},b_{i,3}\}$, and let
$M=(m_{r,c})$. For every block $\overline{B_i}$ of ${\cal F}$ and for every column $M^c=(m_{1,c},m_{2,c},m_{3,c})$ of $M$,  set
$\overline{B_i}\circ M^c=\{b_{i,1}+m_{1,c},b_{i,2}+m_{2,c},b_{i,3}+m_{3,c}\}$. Then
$${\cal F}\circ M:=\{\overline{B_i}\circ M^c \ | \ i\in I; 1\leq c\leq |H|\}$$
is a $(G,L,3,1)$-DF. 

\medskip
Assume that condition (i) holds. 

\noindent
If two elements $\phi_1=b_{i_1,r_1}+m_{r_1,c_1}$ and $\phi_2=b_{i_2,r_2}+m_{r_2,c_2}$ of $\Phi({\cal F}\circ M)$ are
in the same left coset of $\{0,j\}$ in $G$, then we have $$-m_{r_1,c_1}-b_{i_1,r_1}+b_{i_2,r_2}+m_{r_2,c_2}\in\{0,j\}.$$
Reducing modulo $H$ we get $-\overline{b_{i_1,r_1}}+\overline{b_{i_2,r_2}}\in \{\overline0,\overline{j}\}$.
This necessarily implies that $(i_1,r_1)=(i_2,r_2)$ because ${\cal F}$ is $\{\overline0,\overline{j}\}$-resolvable.
Thus, setting $r_1=r_2=r$, we have $-m_{r,c_1}+m_{r,c_2}\in\{0,j\}$. 
It cannot be $-m_{r,c_1}+m_{r,c_2}=j$ since $j$ does not belong to $H$, hence $-m_{r,c_1}+m_{r,c_2}=0$, i.e., 
$m_{r,c_1}=m_{r,c_2}$. This implies that $c_1=c_2$ because $M$ is homogeneous. We conclude that the two triples
$(i_1,r_1,c_1)$ and $(i_2,r_2,c_2)$ coincide, i.e., ${\cal F}\circ M$ is $\{0,j\}$-resolvable.

\medskip
Now assume that condition (ii) holds.

\noindent
For each $i\in I$ there is a suitable translate of $\overline B_i$, say $\overline{B'_i}=\overline{B_i}+\overline{\tau_i}$, such that ${\cal F}$ and
${\cal F}'=\{\overline{B'_i} \ | \ i\in I\}$ are $(G/H,L/H,3,1)$-DFs with
$\Phi({\cal F}) \ \cup \ \Phi({\cal F}')$ a partition of $(G/H) \setminus (L/H)$. 
Note that we can rewrite each $\overline{B'_i}$ in the form $\overline{B'_i}=B_i+t_i$ for a suitable $t_i$ which commutes with $j$.
Indeed we have $\tau_i+j-\tau_i=j_i$ with $j_i$ one of the three involutions of $G$. The fact that $j\in H \trianglelefteq G$
implies that $H$ is pertinent so that there exists $h_i\in H$ such that $h_i+j_i-h_i=j$. Then, setting $t_i=h_i+\tau_i$, it is easy to see that
$\overline{B'_i}=B_i+t_i$ and that $t_i+j=j+t_i$.

Set

\medskip\noindent
${\cal F} \ \scriptsize{\RIGHTcircle}  \ M=\{\overline{B_i}\circ M^c \ | \ i\in I; 1\leq c\leq{|H|\over2}\} \ \cup$

\medskip\hfill
$\{\overline{B_i}\circ M^c+t_i \ | \ i\in I; {|H|\over2}< c\leq|H|\}.$

\medskip\noindent
Of course ${\cal F} \ \scriptsize{\RIGHTcircle} \ M$ is a $(G,H,k,1)$-DF which is strongly equivalent to ${\cal F} \ {\circ} \ M$.
Let us show that it is $\{0,j\}$-resolvable. We have:
$$\Phi({\cal F} \ \scriptsize{\RIGHTcircle} \ M)=\{\phi_{i,r,c} \ | \ i\in I; 1\leq r \leq 3; \ 1\leq c\leq |H|\}$$
with $$\phi_{i,r,c}=\begin{cases}b_{i,r}+m_{r,c} & \mbox{ if $c\leq|H|/2$};
\medskip\cr b_{i,r}+m_{r,c}+t_i & \mbox{ if $c>|H|/2$.}\end{cases}$$
Assume that we have
\begin{equation}\label{-phi+phi}
-\phi_{i_1,r_1,c_1}+\phi_{i_2,r_2,c_2}\in\{0,j\}
\end{equation}
for suitable triples $(i_1,r_1,c_1)$ and $(i_2,r_2,c_2)$. 
We have to prove that  these triples are necessarily equal.
Without loss of generality we can assume that $c_1\leq c_2$. So we have the following three possible cases.

\medskip
\underline{1st case}: $c_1\leq c_2\leq |H|/2$.

\medskip
Here (\ref{-phi+phi}) implies that $-m_{r_1,c_1}-b_{i_1,r_1}+b_{i_2,r_2}+m_{r_2,c_2}\in \{0,j\}$.
Reducing modulo $H$ we get $\overline{b_{i_1,r_1}}=\overline{b_{i_2,r_2}}$ and then $(i_1,r_1)=(i_2,r_2)$ because
${\cal F}$ is disjoint. Thus,
setting $r_1=r_2=r$, we have $-m_{r,c_1}+m_{r,c_2}\in\{0,j\}$ that is possible only for $c_1=c_2$
because $M$ is $\{0,j\}$-splittable. 
We conclude that $(i_1,r_1,c_1)=(i_2,r_2,c_2)$.

\medskip
\underline{2nd case}: $|H|/2< c_1\leq c_2\leq |H|$.

\medskip
Here (\ref{-phi+phi}) implies that $-t_{i_1}-m_{r_1,c_1}-b_{i_1,r_1}+b_{i_2,r_2}+m_{r_2,c_2}+t_{i_2}\in \{0,j\}$.
Reducing modulo $H$ we get $\overline{b_{i_1,r_1}}+\overline{t_{i_1}}=\overline{b_{i_2,r_2}}+\overline{t_{i_2}}$ and then $(i_1,r_1)=(i_2,r_2)$ because
${\cal F}'$ is disjoint. Thus,
setting $i_1=i_2=i$ and $r_1=r_2=r$, we have $-m_{r,c_1}+m_{r,c_2}\in\{0,t_i+j-t_i\}=\{0,j\}$, the last equality being true since $t_i$ commutes with $j$. 
This is possible only for $c_1=c_2$ because $M$ is $\{0,j\}$-splittable. 
We conclude that $(i_1,r_1,c_1)=(i_2,r_2,c_2)$.

\medskip
\underline{3rd case}: $c_1\leq |H|/2$; $c_2> |H|/2$.

\medskip
Here (\ref{-phi+phi}) gives $-m_{r_1,c_1}-b_{i_1,r_1}+b_{i_2,r_2}+m_{r_2,c_2}+t_i\in \{0,j\}$.
Reducing modulo $H$ we get $\overline{b_{i_1,r_1}}=\overline{b_{i_2,r_2}}+\overline{t_i}$.
On the other hand $\overline{b_{i_2,r_2}}$ and $\overline{b_{i_1,r_1}}+\overline{t_i}$ belong to $\Phi({\cal F})$
and $\Phi({\cal F}')$, respectively. This is absurd since $\Phi({\cal F})$ and $\Phi({\cal F}')$
are disjoint.
\end{proof}

As an important consequence of the above theorem we get the following corollaries.

\begin{cor}\label{36Q}
If the components of $n$ are all congruent to $1$ $($mod $4)$,
then there exists a $(G_1\times V_{3n},G_1\times V_3,3,1)$-RDF.
\end{cor}
\begin{proof}
Consider the group $G=G_1\times V_{3n}$ and its subgroups $L=G_1\times V_3$, and $H=G_1\times V_1$.
We have $G/H\simeq V_{3n}$ and $L/H\simeq V_3$, hence there exists a doubly disjoint
$(G/H,L/H,3,1)$-DF by Theorem \ref{DDDF}. There also exists a splittable $(H,3,1)$-DM by Example \ref{(12,3,1)-DM}.
Thus, by Theorem \ref{DF+DM}(ii), there exists a $(G,L,3,1)$-RDF, i.e., a $(G_1\times V_{3n}, G_1\times V_3,3,1)$-RDF.
\end{proof}

\begin{cor}\label{48Q}
If the components of $n$ are all congruent to $1$ $($mod $4)$,
then there exists a $(G_2\times V_{n},G_1\times V_1,3,1)$-RDF.
\end{cor}
\begin{proof}
Consider the group $G=G_2\times V_{n}$ and its subgroups $L=G_1\times V_n$, and $H=\{0\}\times V_n$.
We have $G/H\simeq G_2$ and $L/H\simeq G_1$, hence there exists a doubly disjoint
$(G/H,L/H,3,1)$-DF (see Subsection \ref{51}). There also exists a homogeneous $(H,3,1)$-DM by Theorem \ref{HDM}.
Thus there exists a $(G,L,3,1)$-RDF, i.e., a $(G_2\times V_{n}, G_1\times V_n,3,1)$-RDF by Theorem \ref{DF+DM}(i).
Now recall that there exists a $(G_1\times V_n,G_1\times V_1,3,1)$-RDF by Theorem \ref{15mod24}.
We get the assertion by applying Proposition \ref{matrioska} with the chain $G_2\times V_n\geq G_1\times V_n\geq G_1\times V_1$.
\end{proof}

\section{Main results}

We are finally able to prove the sufficient conditions given by the main Theorem 1.1.

\subsection{$3$-pyramidal $KTS(24n+9)$} 

Recall that Theorem \ref{9mod24} says that there exists a $(D\times V_{4n+1},D\times V_1,3,1)$-RDF 
whenever the prime decomposition of $4n+1$ does not contain primes $p\equiv3$ $($mod $4)$ raised to an odd power.
Equivalently, whenever $4n+1$ is a sum of two squares (see, e.g., \cite{Rosen}).
Thus, considering that $24n+9=6(4n+1)+3$ and that $D\times V_{4n+1}$ is a pertinent group of order $6(4n+1)$, 
we get Theorem \ref{main}(i) by applying  Proposition \ref{matrioska3}.

\begin{thm}\label{main(i)}
If $4n+1$ is a sum of two squares, then
there exists a $3$-pyramidal KTS$(24n+9)$.
\end{thm}

\begin{rem}\label{rem(i)}
Recall that Theorem \ref{9mod24} assures a group of strong
multipliers of order the greatest odd divisor of
$\psi(4n+1)$. Therefore, Proposition \ref{matrioska3} guarantees that
the number of symmetries of each KTS$(24n+9)$
obtainable via Theorem \ref{main(i)} is at least equal to
$(24n+6)m$, where $m$ is the greatest odd divisor of
$\psi(4n+1)$.
\end{rem}

\subsection{$3$-pyramidal $KTS(24n+15)$} 
Here we prove Theorem \ref{main}(ii), that is our main result on $3$-pyramidal $KTS(24n+15)$.

\begin{thm}\label{main(ii)}
There exists a $3$-pyramidal KTS$(24n+15)$ whenever one of the following conditions holds:
\begin{description}
\item[1)] $2n+1$ is divisible by $3$;
\item[2)] the square-free part of $2n+1$ does not contain primes congruent to $11$ $($mod $12)$.
\end{description}
\end{thm}
\begin{proof}
We have $24n+15=12(2n+1)+3$ and $G_1\times V_{2n+1}$ is a pertinent group of order $12(2n+1)$. 
Hence, by Proposition \ref{matrioska3}, it is enough to prove the existence of a $(G_1\times V_{2n+1},G_1\times V_i,3,1)$-RDF with $i=1$ or 3.

\medskip\noindent
\begin{description}
\item[]\underline{Case 1)}: $2n+1\equiv0$ (mod 3).
\end{description}

Set $e=2$ if $9$ is a component of $2n+1$, otherwise set $e=1$. Now let $(2n+1)/3^e = PQ$ 
where $P$ and $Q$ are the product of all the components of $(2n+1)/3^e$ congruent to 3 
and 1 (mod 4), respectively. Note that $3$ is not a component of $P$, otherwise
$3^{e+1}=
\begin{cases}
9  & \text{if $e=1$}\\
27 & \text{if $e=2$}
\end{cases}$
would be a component of $2n+1$, contradicting the definition of the integer $e$.

Consider the group $G=G_1\times V_{2n+1}$ and its subgroups $L=G_1\times V_{3^eQ}$ and 
$H=\{0\}\times V_{Q}$.
Since $G/H\simeq G_1\times V_{3^eP}$ and $L/H\simeq G_1\times V_{3^e}$, 
by Corollaries \ref{36P} and \ref{12*9P}, there exists
a $(G/H,L/H,3,1)$-RDF. Also, there exists a homogeneous $(H,3,1)$-DM by Theorem \ref{HDM}.
Thus, by Theorem \ref{DF+DM}(i), there exists a $(G, L, 3, 1)$-RDF.
There is also a $(L, G_1\times V_{3^{2-e}},3,1)$-RDF by 
Theorem \ref{15mod24} (when $e=2$)
and Corollary \ref{36Q} (when $e=1$). Therefore, by applying Proposition \ref{matrioska} 
with the chain $G \geq L \geq G_1\times V_{3^{2-e}}$
we get a $(G, G_1\times V_{3^{2-e}}, 3, 1)$-RDF.

\medskip\noindent
\begin{description}
\item[]\underline{Case 2)}: the square-free part of $2n+1$ does not contain any prime congruent to $11$ (mod 12).
\end{description}

\noindent
We can assume that $2n+1$ is not divisible by 3 in view of Case 1). Thus, by assumption,
we can write $2n+1=PQ$ where $P$ is the product of all components of $2n+1$ that are congruent to 7 modulo 12
and $Q$ is the product of all components of $n$ that are congruent to 1 modulo 4.
Of course it is understood that $P$ and/or $Q$ may be equal to 1 in the case that the components of
the respective kinds do not exist. 
If $P=1$ or $Q=1$, we have a $(G_1\times V_{2n+1},G_1\times V_1,3,1)$-RDF by Theorem \ref{15mod24}
or Theorem \ref{15mod24bis}, respectively.
If both $P$ and $Q$ are greater than 1, consider the group $G=G_1\times V_{2n+1}$ and its subgroups $L=G_1\times V_P$ and $H=\{0\}\times V_P$.
We have $G/H\simeq G_1\times V_Q$ and $L/H\simeq G_1\times V_1$.
Thus there exists a $(G/H,L/H,3,1)$-RDF by Theorem \ref{15mod24}. Also, there exists a homogeneous $(H,3,1)$-DM
by Theorem \ref{HDM}. It follows,  by Theorem \ref{DF+DM}(i), that there exists a $(G,L,3,1)$-RDF,
i.e. a $(G_1\times V_{2n+1},G_1\times V_P,3,1)$-RDF. We also have a $(G_1\times V_P,G_1\times V_1,3,1)$-RDF
by Theorem \ref{15mod24bis}. Applying Proposition \ref{matrioska} with the chain $G_1\times V_{2n+1} \geq G_1\times V_P \geq G_1\times V_1$
we get a $(G_1\times V_{2n+1},G_1\times V_1,3,1)$-RDF.
\end{proof}


\begin{rem}\label{rem(ii)}
We recall that Theorems \ref{15mod24}
and \ref{15mod24bis}, and Corollaries \ref{36P} and \ref{12*9P}
show the existence of a group of strong multipliers. Therefore, it is not difficult to check
that Propositions \ref{matrioska}
and \ref{matrioska3} guarantee that
\begin{enumerate}
  \item the number of symmetries of each KTS$(72n'+39)$ obtainable via Theorem \ref{main(ii)}.(1)
 is at least equal to $m(72n'+39)$, where $m$ is the greatest odd divisor of $\psi(P)$ and $P>1$ is the product of all the components of $2n'+1$ congruent to $7$ or $11$ (mod 12);
  \item the number of symmetries of each KTS$(24n+15)$ built in Theorem \ref{main(ii)}.(2)
 is at least equal to $m(24n+15)$ where $m$ is defined as follows:
 \begin{enumerate} 
 \item if $Q>1$ is the product of all the components of $2n+1$ congruent to 1 (mod 4), then $m$ is
the greatest odd divisor of $\psi(Q)$;
 \item if all the components of $2n+1$ are congruent to 7 (mod 12), then $m$ is
 the greatest odd divisor of $\psi(2n+1)$ coprime with $6$.
 \end{enumerate}
\end{enumerate}
\end{rem}

\subsection{$3$-pyramidal $KTS(48n+3)$}
In this subsection we will prove that the necessary condition for the existence of a KTS$(v)$ is also sufficient
when $v\equiv3$ (mod 48), that is Theorem \ref{main}(iii).

We recall that $\{G_\alpha: \alpha\geq1\}$ is the series of pertinent groups considered in Section 2. 
For $0\leq i\leq \alpha-1$,  the subgroup of $G_\alpha$ with underlying-set $\Z_3\times2^i\Z_{2^{\alpha}}\times2^i\Z_{2^{\alpha}}$ 
is isomorphic to $G_{\alpha-i}$. Hence, by abuse of notation, this subgroup will be denoted by $G_{\alpha-i}$ in the following.

\begin{thm}\label{main(iii)}
There exists a KTS$(4^e48n+3)$ for every non-negative integer $e$ and every positive odd integer $n$.
\end{thm}
\begin{proof}
Set $\alpha=e+2$ and note that $G_\alpha\times V_n$ is a pertinent group of order $4^e48n$.
Hence, by Proposition \ref{matrioska3}, it is enough to prove the existence of a 
 $(G_{{\alpha}}\times V_n,G_1\times V_i,3,1)$-RDF with $i=1$ or 3 for any $\alpha\geq2$ and any odd $n\geq1$.
 
We distinguish five cases.

\medskip
\underline{1st case}: $n=1$.

\smallskip\noindent
Let us prove the existence of a $(G_{{\alpha}},G_1,3,1)$-RDF for every $\alpha\geq2$.
A $(G_2,G_1,3,1)$-RDF has been given in Subsection \ref{51} and a
$(G_3,G_2,3,1)$-RDF can be found in Appendix \ref{G3,G2,3,1RDF}. 
Now let $\alpha\geq4$ and assume, by induction, that there exists a $(G_{{\beta}},G_{\beta-1},3,1)$-RDF for $2\leq\beta<\alpha$.
Set $G=G_{\alpha}$, $L=G_{\alpha-1}$, and let $H$ be the subgroup of $G$ with underlying set
$\{0\}\times2^{\alpha-2}\Z_{2^{\alpha}}\times2^{\alpha-2}\Z_{2^{\alpha}}$. 
Note that $H$ is isomorphic to $\Z_4\times\Z_4$ so that there exists a splittable $(H,3,1)$-DM
by Example \ref{(16,3,1)-DM}. The quotient groups $G/H$ and $L/H$ are isomorphic 
to $G_{\alpha-2}$ and $G_{\alpha-3}$, respectively.
Thus, by the induction hypothesis, there exists a $(G/H,L/H,3,1)$-RDF.
Applying Theorem \ref{DF+DM}(ii) we get a $(G,L,3,1)$-RDF, i.e., a $(G_{\alpha},G_{\alpha-1},3,1)$-RDF.
Applying Proposition \ref{matrioska} with the chain $G_\alpha \geq G_{\alpha-1} \geq \dots \geq G_2 \geq G_1$
we get a $(G_{{\alpha}},G_1,3,1)$-RDF.

\medskip
\underline{2nd case}: $n=3$.

\smallskip\noindent
A $(G_2\times V_3,G_1\times V_3,3,1)$-RDF will be given in Appendix \ref{G2xV3,G1xV3,3,1RDF}. 
Let $\alpha\geq3$ and let $\beta$ be any integer of the closed interval $[3,\alpha]$.
Consider the group $G=G_{\beta}\times V_3$ and its subgroups $L=G_{\beta-1}\times V_3$ and $H=K\times V_3$
where $K$ is the Klein subgroup of $G_\beta$.
We have $G/H\simeq G_{\beta-1}$ and $L/H\simeq G_{\beta-2}$ so that there exists a $(G/H,L/H,3,1)$-RDF (see 1st case).
Also, $H\simeq\Z_2\times\Z_6$ so that there exists a splittable $(H,3,1)$-DM
by Example \ref{(12,3,1)-DMbis}. Thus, by Theorem \ref{DF+DM}(ii),  there exists a $(G,L,3,1)$-RDF,
i.e., a $(G_\beta\times V_3,G_{\beta-1}\times V_3,3,1)$-RDF.

Applying Proposition \ref{matrioska} with the chain $$G_\alpha\times V_3 \geq G_{\alpha-1}\times V_3 \geq \dots \geq G_2\times V_3 \geq G_1\times V_3$$
we get a $(G_{{\alpha}}\times V_3,G_1\times V_3,3,1)$-RDF.

\medskip
\underline{3rd case}: $3<n\equiv0$ (mod 3). 

\smallskip\noindent
Consider the group $G=G_\alpha\times V_n$ and its subgroups $L=G_1\times V_n$ and $H=\{0\}\times V_n$.
We have $G/H\simeq G_\alpha$ and $L/H\simeq G_1$ so that there exists a $(G/H,L/H,3,1)$-RDF (see first case).
There also exists a homogeneous $(H,3,1)$-DM by Theorem \ref{HDM}. It follows that there exists a 
$(G,L,3,1)$-RDF, i.e., a $(G_\alpha\times V_n,G_1\times V_n,3,1)$-RDF by Theorem \ref{DF+DM}(i). From the proof of Case 1) of
Theorem \ref{main(ii)} we also have a  $(G_1\times V_n,G_1\times V_i,3,1)$-RDF with $i=1$ or $3$. Thus we have a
$(G_\alpha\times V_n,G_1\times V_i,3,1)$-RDF by Proposition \ref{matrioska}.

\medskip
\underline{4th case}: $1<n\not\equiv0$ (mod 3) and $\alpha=2$.

\smallskip\noindent
Write $n=PQ$ where $P$ is the product of all the components of $n$ congruent to $3$ (mod 4)
and $Q$  is the product of all the components of $n$ congruent to $1$ (mod 4).
If $Q=1$, we get the required $(G_2\times V_n,G_2\times V_1,3,1)$-RDF from Corollary \ref{48P}.
If $Q>1$, consider the group $G=G_2\times V_n$ and its subgroups $L=G_2\times V_Q$ and $H=\{0\}\times V_Q$.
We have $G/H\simeq G_2\times V_P$ and $L/H\simeq G_2$ so that there exists a $(G/H,L/H,3,1)$-RDF either trivially if $P=1$,
or by Corollary \ref{48P} if $P>1$.
There also exists a homogeneous $(H,3,1)$-DM by Theorem \ref{HDM}. It follows that there exists a 
$(G,L,3,1)$-RDF, i.e., a $(G_2\times V_n,G_2\times V_Q,3,1)$-RDF by Theorem \ref{DF+DM}(i).  We also
have a  $(G_2\times V_Q,G_2\times V_1,3,1)$-RDF because of Corollary \ref{48Q} and, from the first case, 
a $(G_2\times V_1,G_1\times V_1,3,1)$-RDF. Thus, applying Proposition \ref{matrioska} with the chain 
$G_2\times V_n\geq G_2\times V_Q\geq G_2\times V_1\geq G_1\times V_1$  
we finally get a $(G_2\times V_n,G_1\times V_1,3,1)$-RDF.

\medskip
\underline{5th case}: $n>3$ and $\alpha>2$.

\smallskip\noindent
Let $\alpha\geq3$ and let $2\leq\beta\leq\alpha$.
Consider the group $G=G_{\beta}\times V_n$ and its subgroups $L=G_{\beta-1}\times V_n$ and $H=\{0\}\times V_n$.
We have $G/H\simeq G_\beta$ and $L/H\simeq G_{\beta-1}$ so that there exists a $(G/H,L/H,3,1)$-RDF
(see 1st case). Also, there exists a homogeneous $(H,3,1)$-DM by Theorem \ref{HDM}.
Thus  there exists a $(G,L,3,1)$-RDF, i.e., a $(G_\beta\times V_n,G_{\beta-1}\times V_n,3,1)$-RDF by Theorem \ref{DF+DM}(i).
Applying Proposition \ref{matrioska} with the chain $$G_\alpha\times V_n \geq G_{\alpha-1}\times V_n \geq \dots \geq G_2\times V_n$$
we get a $(G_{{\alpha}}\times V_n,G_2\times V_n,3,1)$-RDF. From either the third or the fourth case we also have a $(G_2\times V_n,G_1\times V_i,3,1)$-RDF
with $i=1$ or $3$. Then, by Proposition \ref{matrioska} again, we have a $(G_{{\alpha}}\times V_n,G_1\times V_i,3,1)$-RDF with $i=1$ or 3.
\end{proof}



\begin{rem}\label{rem(iii)}
We recall that Corollary \ref{48P}
shows the existence of a group of strong multipliers. Therefore, it is not difficult to check
that Propositions \ref{matrioska}
and \ref{matrioska3} guarantee that
the number of symmetries of each KTS$(48n+3)$ obtainable via Theorem \ref{main(iii)}, when $n$ is not divisible by $3$, is at least equal to $m(48n+3)$, where $m$ is the greatest odd divisor of 
$\psi(P)$ and $P>1$ is the product of all the components of $n$ congruent to $3$ (mod 4).
\end{rem}

\section{Open problems}

The problem of classifying the 3-pyramidal KTS$(v)$ remains open in the following cases.
\begin{itemize}
\item $v-3=24n+6$ and $4n+1$ is not a sum of two squares;
\item $v-3=72n\pm12$ and its prime decomposition contains a prime factor $p\equiv11$ $($mod $12)$ raised to an odd power;
\end{itemize}
The open cases above could be closed if one solves the following problems, respectively.

\begin{prob}
Determine a $(D\times V_{pq},D\times V_1,3,1)$-RDF for every pair $(p,q)$ of distinct primes congruent to $3$ modulo $4$.
\end{prob}

\begin{prob}
Determine a $(G_1\times V_p,G_1\times V_1,3,1)$-RDF for every prime $p\equiv 11$ modulo $12$.
\end{prob}

Our research naturally leads to consider also the following collateral problems which, in our opinion, are interesting on their own.

\begin{prob}
Determine the admissible groups $H$ for which there exists a splittable $(H,3,1)$ difference matrix.
\end{prob}

\begin{prob}
Determine the set of all values of $n$ for which there exists a doubly disjoint $(\Z_3\times V_{2n+1},\Z_3\times V_1,3,1)$ difference family.
\end{prob}

\bigskip
\footnotesize
\noindent\textit{Acknowledgments.}
The authors gratefully acknowledge support from GNSAGA of Istituto Nazionale di Alta Matematica.

\bigskip\Large
\begin{appendix}{\bf Appendix}
\normalsize

\section{Pseudo-resolvable $(G_1\times V_9,\{2^3,3\},3,1)$-DF}\label{G1xV9,3,1PRDF}
Let $\Sigma$ be the $\{2^3,3\}$-partial spread of $G_1\times V_9$ 
whose member of order $3$ is $\{0,x,-x\}$ with $x=(1,0,0,0,0)$.
The following seventeen triples of $G_1\times V_9$ 
\small
 $$\{(0, 0, 0, 0, 1),(0, 1, 1, 2, 0),(1, 0, 1, 1, 1)\}, \quad \{(0, 0, 0, 0, 2),(2, 0, 0, 2, 0),(2, 1, 0, 0, 0)\},$$
 $$\{(0, 0, 0, 1, 0),(0, 1, 0, 2, 0),(1, 0, 0, 1, 2)\}, \quad \{(0, 0, 0, 1, 1),(2, 0, 0, 2, 1),(2, 0, 1, 2, 0)\},$$
 $$\{(0, 0, 0, 1, 2),(0, 1, 0, 1, 0),(2, 0, 1, 1, 0)\}\,\quad \{(0, 0, 0, 2, 1),(0, 1, 0, 0, 2),(1, 0, 1, 1, 2)\},$$
 $$\{(0, 0, 0, 2, 2),(0, 0, 1, 0, 1),(2, 0, 0, 1, 2)\}, \quad \{(0, 0, 1, 1, 2),(1, 1, 1, 2, 0),(2, 1, 0, 2, 2)\},$$
 $$\{(0, 0, 1, 2, 2),(1, 0, 1, 1, 0),(2, 0, 0, 2, 2)\}, \quad \{(0, 1, 0, 1, 1),(2, 0, 1, 0, 1),(2, 1, 1, 0, 0)\},$$
$$\{(0, 1, 0, 2, 1),(1, 1, 1, 2, 2),(2, 0, 0, 0, 1)\}, \quad \{(1, 0, 1, 0, 0),(1, 1, 1, 1, 0),(2, 1, 0, 1, 1)\},$$
$$\{(1, 0, 1, 0, 1),(1, 0, 1, 2, 1),(1, 0, 1, 2, 2)\}, \quad \{(1, 1, 0, 0, 2),(2, 0, 0, 1, 1),(2, 0, 1, 0, 2)\},$$
$$\{(1, 1, 0, 2, 0),(1, 1, 1, 1, 1),(2, 0, 1, 1, 2)\}, \quad \{(1, 1, 1, 0, 1),(2, 0, 1, 2, 1),(2, 1, 1, 0, 2)\},$$
$$\{(1, 1, 1, 0, 2),(1, 1, 1, 2, 1),(2, 1, 1, 1, 0)\}$$

\normalsize\noindent
are the blocks of a $(G_1\times V_9,\Sigma,3,1)$-PRDF.

\section{Pseudo-resolvable $(G_2,\{2^3,3\},3,1)$-DF}\label{G2,3,1PRDF}
Let $\Sigma$ be the $\{2^3,3\}$-partial spread of $G_2$ 
whose member of order $3$ is $\{0,x,-x\}$ with $x=(1,0,0)$.
The following seven triples of $G_2$ 
$$\{(0, 0, 1), (0, 3, 1), (2, 1, 2)\},\quad\{(0, 1, 0), (1, 0, 3), (2, 3, 2)\},$$
$$\{(0, 1, 1), (1, 3, 3), (2, 0, 2)\},\quad\{(0, 1, 2), (1, 0, 2), (2, 1, 3)\},$$
$$\{(0, 2, 1),(2, 0, 0), (2, 0, 1)\},\quad\{(1, 1, 0), (1, 2, 3), (1, 3, 1)\},$$
$$\{(1, 1, 2), (2, 0, 3), (2, 3, 3)\},$$

\noindent
are the blocks of a pseudo-resolvable $(G_2,\Sigma,3,1)$-DF.

\section{$(G_2\times V_3,G_1\times V_3,3,1)$-RDF}\label{G2xV3,G1xV3,3,1RDF}
The following eighteen triples of $G_2\times V_3$
\small
$$\{(0, 0, 1, 0), (1, 3, 1, 2), (2, 1, 0, 1)\},\quad \{(0, 0, 1, 1), (0, 1, 0,  2), (1, 3, 3, 2)\},$$ 
$$\{(0, 0, 1, 2), (0, 3, 3, 1), (1, 1, 2, 0)\},\quad \{(0, 0, 3, 0), (2, 3, 1, 2), (2, 3, 2, 0)\},$$ 
$$\{(0, 0, 3, 1), (0, 1, 3, 0), (1, 1, 0, 0)\}, \quad \{(0, 1, 0, 0), (1, 3, 1, 1), (2, 0, 1, 0)\},$$ 
$$\{(0, 1, 0, 1), (1, 0, 3, 1), (2, 3, 1, 1)\}, \quad \{(0, 1, 1, 0), (2, 2, 1, 2), (2, 3, 0, 1)\},$$ 
$$\{(0, 1, 2, 0), (0, 2, 1, 2), (1, 3, 3, 0)\}, \quad \{(0,1, 2, 2), (0, 3, 1, 1), (2, 0, 1, 1)\},$$ 
$$\{(0, 1, 3, 2), (1, 0, 1, 2), (2, 3, 0, 0)\}, \quad \{(0, 3, 0, 1), (1, 0, 3, 2), (2, 1, 3, 0)\},$$ 
$$\{(0, 3, 3, 2), (1, 0, 3, 0), (2, 3, 0, 2)\}, \quad \{(1, 0, 1, 0), (1, 1, 2, 2), (1, 1, 3, 0)\},$$
$$\{(1, 0, 1, 1), (1, 1, 1, 1), (1, 1, 2, 1)\}, \quad \{(1, 1, 0, 1), (2, 0, 3, 0), (2, 1, 1, 0)\},$$ 
$$\{(1, 1, 0, 2), (2, 1, 1, 2), (2, 2, 1, 1)\}, \quad \{(2, 1, 0, 2), (2, 1, 1, 1), (2, 2, 3, 2)\}$$

\noindent\normalsize
are the blocks of a $(G_2\times V_3,G_1\times V_3,3,1)$-RDF.

\normalsize

\section{$(G_3,G_2,3,1)$-RDF}\label{G3,G2,3,1RDF}
The following twentyfour triples of $G_3$
$$\{(0, 0, 1), (0, 5, 2), (0, 7, 5)\}, \quad \{(0, 0, 3), (2, 1, 1), (2, 5, 2)\},$$
$$\{(0, 0, 5), (2, 1, 7), (2, 3, 6)\},\quad \{(0, 0, 7), (0, 1, 1), (2, 7, 0)\},$$
$$\{(0, 1, 0), (0, 7,  3), (2, 6, 1)\},\quad \{(0, 1, 4), (1, 4, 7), (2, 7, 5)\},$$ 
$$\{(0, 1, 5), (0, 5, 6), (2, 6, 7)\}, \quad \{(0, 1,  7), (1, 3, 2), (2, 4, 5)\},$$
$$\{(0, 2, 1), (2, 3,  5), (2, 5, 0)\},\quad\{(0, 2, 5), (1, 1, 4), (1, 1,  5)\},$$
$$\{(0, 3, 0), (1, 1, 1), (1, 6, 1)\}, \quad \{(0, 3,  3), (2, 0, 3), (2, 1, 0)\},$$
$$\{(0, 3, 5), (1, 3,  6), (2, 2, 1)\}, \quad \{(0, 3, 6), (1, 2, 3), (1, 3,  5)\},$$
$$\{(0, 5, 7), (0, 7, 0), (1, 4, 5)\},\quad\{(0, 6, 3), (1, 1, 2), (2, 3, 7)\},$$
$$\{(0, 6, 7), (1, 3, 3), (1, 5, 2)\}, \quad \{(0, 7, 6), (2, 6, 3), (2, 7, 7)\},$$
$$\{(1, 1, 0), (1, 4, 3), (2, 1, 5)\}, \quad \{(1, 1, 3), (1, 4, 1), (2, 7, 4)\},$$
$$\{(1, 1, 7), (2, 1, 2), (2, 4, 3)\},\quad \{(1, 3, 7), (2, 4, 1), (2, 7, 6)\},$$
$$\{(1, 6, 3), (1, 7, 4), (2, 5, 7)\}, \quad \{(1, 6, 5), (1, 7, 0), (1, 7, 5)\}$$

\noindent\normalsize
are the blocks of a $(G_3,G_2,3,1)$-RDF.

\end{appendix}


\begin{thebibliography}{99}

\bibitem{AFL}
I.J. Anderson, N.J. Finizio, P.A. Leonard,
{\it New product theorems for $Z$-cyclic whist tournaments},
J. Combin. Theory Ser. A {\bf88} (1999), 162--166.

\bibitem{BJL} T. Beth, D. Jungnickel and H. Lenz,
Design Theory. Cambridge University Press, Cambridge, 1999.


\bibitem{BBRT} S. Bonvicini, M. Buratti, G. Rinaldi, T. Traetta,
{\it Some progress on the existence of $1$-rotational Steiner triple systems},
Des. Codes Cryptogr. {\bf 62} (2012), 63--78.

\bibitem{recursive} M. Buratti,
{\it Recursive constructions for difference matrices 
and relative difference families}, 
J. Combin. Des. {\bf 6} (1998), 165--182.

\bibitem{B99} M. Buratti,
{\it Old and new designs via difference multisets and strong difference families}, 
J. Combin. Des. {\bf 7} (1999), 406-425.

\bibitem{B} M. Buratti,
{\it $1$-rotational Steiner Triple Systems over arbitrary groups},
J. Combin. Des. {\bf 9} (2001), 215-226.



\bibitem{disjoint}
M. Buratti, {\it On disjoint $(v,k,k-1)$ difference families},
Des. Codes Cryptogr. {\bf87} (2019), 745--755.


\bibitem{BuGhi} M. Buratti, D. Ghinelli,
{\it On disjoint $(3t,3,1)$ cyclic difference families},
J. Statist. Plann. Inference {\bf 140} (2010), 1918--1922.

\bibitem{BuGio} M. Buratti, L. Gionfriddo,
{\it Strong difference families over arbitrary groups},
J. Combin. Des. {\bf16} (2008), 443-461.

\bibitem{BP} M. Buratti, A. Pasotti,
{\it Combinatorial designs and the theorem of Weil
on multiplicative character sums}, Finite Fields Appl.  {\bf 15}, (2009), 332--344.

\bibitem{BN} M. Buratti, A. Nakic,
{\it Transrotional Kirkman triple systems}, in preparation.

\bibitem{BRT} M. Buratti, G. Rinaldi, T. Traetta,
{\it $3-$pyramidal Steiner triple systems}, Ars Math. Contemp. {\bf 13}, (2017), 95--106.


\bibitem{BZ} M. Buratti, F. Zuanni,
{\it $G$-invariantly resolvable Steiner $2-$designs which are $1$-rotational over $G$},
Bull. Belg. Math. Soc. {\bf 5}, (1998), 221--235.

\bibitem{BZ01} M. Buratti, F. Zuanni,
{\it Explicit constructions for $1$-rotational Kirkman triple systems},
Utilitas Math. {\bf 59}, (2001), 27--30.



\bibitem{CCFW} Y. Chang, S. Costa, T. Feng, X. Wang, 
{\it Strong difference families of special types},
Discrete Math. {\bf343}, (2020).

\bibitem{CC}
M.J. Colbourn, C.J. Colbourn,
{\it Recursive constructions for cyclic block designs}, J. Statist. Plann. Infer. {\bf10} (1984), 97--103.

\bibitem{CD} C.J. Colbourn, J.H. Dinitz, 
Handbook of Combinatorial Designs.
Second Edition, Chapman \& Hall/CRC, Boca Raton, FL, 2006

\bibitem{CR} C.J. Colbourn, A. Rosa,
{\it Triple systems}, Clarendon Press, Oxford, 1999.

\bibitem{CFW1} S. Costa, T. Feng, X. Wang, 
{\it New $2$-designs from strong difference families},
Finite Fields Appl.  {\bf50} (2018), 391--405.

\bibitem{CFW2} S. Costa, T. Feng, X. Wang, 
{\it Frame difference families and resolvable balanced incomplete block designs}, 
Des. Codes Cryptogr. {\bf86} (2018), 2725--2745.

\bibitem{CKZ} A. Custic, V. Krcadinac, Y. Zhou,
{\it Tiling groups with difference sets},
Electron.  J.  Combin. {\bf22} (2015), $\sharp$P2.56

\bibitem{DR} J.H. Dinitz, P. Rodney, {\it Block disjoint difference families for Steiner triple systems}, Utilitas Math. {\bf52} (1997), 153--160.

\bibitem{DS} J.H. Dinitz, N. Shalaby, {\it Block disjoint difference families for Steiner triple systems: 
$v\equiv 3$ $($mod $6)$}, J. Statist. Plann. Inference {\bf106} (2002), 77--86.

\bibitem{D} J.H. Dinitz, H.C. Williams, Number Theory and Finite Fields. In: Colbourn, C.J., Dinitz, J.H. (eds.) Handbook of Combinatorial Designs. 2nd
edn., pp. 791--818. Chapman \& Hall/CRC, Boca Raton (2006).

\bibitem{Doyen} J. Doyen,
{\it A note on reverse Steiner triple systems},
Discrete Math. {\bf1} (1972), 315--319.

\bibitem{Evans} A.B. Evans,
{\it The admissibility of sporadic simple groups},
J. Algebra {\bf321} (2009), 105--116.

\bibitem{FP} G. Falcone, M. Pavone,
{\it Kirkman's tetrahedron and the fifteen schoolgirl problem},
Amer. Math. Monthly {\bf 118} (2011), 887--900.

\bibitem{Gardner}
R.B. Gardner,
{\it Steiner triple systems with transrotational automorphisms},
Discrete Math. {\bf131} (1994), 99--104.
 
\bibitem{Ge} G. Ge,
{\it On $(g,4;1)$ difference matrices},
Discrete Math. {\bf301} (2005), 164--174.

\bibitem{GMJ} M. Genma, M. Mishima, M. Jimbo, 
{\it Cyclic resolvability of cyclic Steiner $2-$designs},
J. Combin. Des. {\bf 5}, (1997), 177--187.

\bibitem{HP} M. Hall, L.J. Paige,
{\it Complete mappings of finite groups}
Pacific J. Math., {\bf5} (1955), 541--549.

\bibitem{HRW}
H. Hanani, D.K. Ray.Chaudhuri, R.M. Wilson,
{\it On resolvable designs}, Discrete Math. {\bf3} (1972), 343--357.

\bibitem{Hawkins}
T. Hawkins, {\it The Erlanger Program of Felix Klein: Reflections on Its Place In the History of Mathematics}, 
Historia Mathematica {\bf11} (1984), 442--470.

\bibitem{Janko}
Z. Janko,
{\it A classification of finite $2$-groups with exactly three involutions},
J. Algebra {\bf291} (2005), 505--533.

\bibitem{Jung}
D. Jungnickel, {\it Composition theorems for difference families and regular planes},
Discret. Math. {\bf23} (1978), 151--158.


\bibitem{KM}
S. Kageyama, Y. Miao,
{\it A construction for resolvable designs and its generalizations},
Graphs Combin. {\bf14} (1998), 11--24.

\bibitem{K} P. Keevash, {\it The existence of designs}. Preprint,
arXiv:1401.3665.

\bibitem{Kirkman47}
T. P. Kirkman, 
On a problem in combinations, Cambridge and Dublin Math. J. {\bf2} (1847), 191--204.

\bibitem{Konvisser}
M.W. Konvisser, {\it $2$-Groups which Contain Exactly Three Involutions},
Math. Z. {\bf130}, (1973) 19--30. 

\bibitem{L} G. Lovegrove,
 {\it The automorphism groups of Steiner triple systems obtained by the Bose construction},
 J. Algebr. Comb. {\bf18} (2003), 159--170.


\bibitem{Mend} E. Mendelsohn,
{\em On the groups of automorphisms of Steiner triple and quadruple systems},
J. Combin. Theory Ser. A {\bf 25} (1978), 97-104.

\bibitem{MendRosa}
E. Mendelsohn, A. Rosa, {\it One-factorizations of the complete graph -- A survey}, J. Graph
Theory {\bf9} (1985), 43--65.

\bibitem{MR} M. Meszka, A. Rosa, {\it Cyclic Kirkman triple systems}, Congr. Numer. {\bf 188} (2007), 129--136.

\bibitem{M} M. Mishima,
{\it The spectrum of $1$-rotational Steiner triple systems over a dicyclic group},
Discrete Math. {\bf 308} (2008), 2617--2619.

\bibitem{Momihara} K. Momihara, 
{\it Strong difference families, difference covers, and their applications for relative difference families}, 
Des. Codes Cryptogr., {\bf51} (2008), 253--273.

\bibitem{PanChang}
R. Pan, Y. Chang,
{\it A note on difference matrices over non-cyclic finite abelian groups},
Discrete Math. {\bf339} (2016), 822--830.

\bibitem{P} R. Peltesohn,
\emph{Eine Losung der beiden Heffterschen Differenzenprobleme},
Compos. Math. \textbf{6} (1938), 251--257.

\bibitem{PR} K.T. Phelps, A. Rosa,
{\it Steiner triple systems with rotational automorphisms},
Discrete Math. {\bf 33} (1981), 57--66.

\bibitem{RCW} D.K. Ray-Chaudhuri, R. M. Wilson,
{\it Solution of Kirkman's Schoolgirl Problem} Combinatorics, Proc. Sympos. Pure Math., Univ. California Los Angeles 1968 {\bf 19} (1971), 187--203.


\bibitem{Rosa} A. Rosa,
{\it On reverse Steiner triple systems},
Discrete Math. {\bf2} (1972), 61--71.

\bibitem{Rosen}
K.H. Rosen. Elementary Number Theory and its Applications, Sixth Edition. Pearson, 2000.

\bibitem{S}
D.R. Stinson, {\it Frames for Kirkman triple systems}, Discrete Math. {\bf65} (1987), 289--300.

\bibitem{T} L. Teirlinck,
{\it The existence of reverse Steiner triple systems},
Discrete Math. {\bf 6} (1973), 220--245.

\bibitem{W1} R.M. Wilson,
{\it An existence theory for pairwise balanced designs, I: Composition theorems and morphisms},
J. Combin. Theory Ser. A {\bf13} (1972), 246--273.

\bibitem{W2} R.M. Wilson,
{\it An existence theory for pairwise balanced designs, II: The structure
of PBD-closed sets and the existence conjectures},
J. Combin. Theory Ser. A {\bf13} (1972), 71--79.

\bibitem{W3} R.M. Wilson,
{\it An existence theory for pairwise balanced designs, III: Proof of the existence conjectures},
J. Combin. Theory Ser. A {\bf18} (1975), 71--79.


\bibitem{YYL} J. Ying, X. Yang, Y. Li,
{\it Some $20$-regular CDP$(5, 1; 20u)$ and their applications},
Finite Fields Appl.  {\bf17} (2011), 317--328.
\end{thebibliography}
\end{document}